\documentclass[final,leqno,onefignum,onetabnum]{siamltex1213}
\usepackage{amsmath,amssymb,mathdots,array,booktabs,graphicx,epstopdf,mathrsfs}
\usepackage{algpseudocode,algorithm,algorithmicx}
\usepackage{color}
\usepackage[T1]{fontenc}

\newcommand{\bsmat}{\left[\begin{smallmatrix} }
\newcommand{\esmat}{\end{smallmatrix}\right] }

\newcommand{\tn}{\mathbb{T}_n}
\newcommand{\rnn}{\mathbb{R}^{n\times n}}
\newcommand{\cnn}{\mathbb{C}^{n\times n}}
\newcommand{\on}{\mathbb{O}_n}
\newcommand{\un}{\mathbb{U}_n}
\newcommand{\sn}{\mathbb{S}_n}
\newcommand{\hn}{\mathbb{H}_n}
\newcommand{\grn}{\mathbb{GL}(n,\mathbb{R})}
\newcommand{\gcn}{\mathbb{GL}(n,\mathbb{C})}
\newcommand{\an}{\mathbb{A}_n}
\newcommand{\wn}{\mathbb{W}_n}
\newcommand{\na}{\mathscr{N}(\mathcal{A})}
\newcommand{\nad}{\mathscr{N}(\mathcal{A};\delta)}

\DeclareMathOperator{\tr}{tr}
\DeclareMathOperator{\argmin}{argmin}
\DeclareMathOperator{\Bdiag}{Bdiag}
\DeclareMathOperator{\OffBdiag}{OffBdiag}
\DeclareMathOperator{\myvec}{vec}
\DeclareMathOperator{\reshape}{reshape}
\DeclareMathOperator{\sep}{sep}

\DeclareMathOperator{\SNR}{SNR}

\title{An Algebraic Approach to  Non-Orthogonal General Joint Block Diagonalization
\footnotemark[1]}

\author{
Yunfeng Cai
\footnotemark[2]
\and
Chengyu Liu
\footnotemark[3]
}

\begin{document}
\maketitle

\renewcommand{\thefootnote}{\fnsymbol{footnote}}

\footnotetext[1]{This research was supported by NSFC under grants 11301013, 11671023 and 11421101.}
\footnotetext[2]{LMAM \& School of Mathematical Sciences,
Peking University, Beijing, 100871, P.R. China, \email{yfcai@math.pku.edu.cn}}
\footnotetext[3]{LMAM \& School of Mathematical Sciences,
Peking University, Beijing, 100871, P.R. China, \email{liuchyu@pku.edu.cn}}

\renewcommand{\thefootnote}{\arabic{footnote}}

\slugger{mms}{xxxx}{xx}{x}{x--x}

\begin{abstract}
The exact/approximate non-orthogonal general joint block diagonalization ({\sc nogjbd}) problem
of a given real matrix set $\mathcal{A}=\{A_i\}_{i=1}^m$
is to find a nonsingular matrix $W\in\mathbb{R}^{n\times n}$ (diagonalizer) such that
$W^T A_i W$ for $i=1,2,\dots, m$ are all exactly/approximately block diagonal matrices
with the same diagonal block structure and with as many diagonal blocks as possible.
In this paper, we show that 
a solution to the exact/approximate {\sc nogjbd} problem can be obtained
by finding the exact/approximate solutions to the system of linear equations 
$A_iZ=Z^TA_i$ for $i=1,\dots, m$,
followed by a block diagonalization of $Z$ via similarity transformation.
A necessary and sufficient condition for the equivalence of the solutions to the exact  {\sc nogjbd} problem is established.
Two numerical methods are proposed to solve the {\sc nogjbd} problem,
and numerical examples are presented to show the merits of the proposed methods.
\end{abstract}

\vskip 2mm



\pagestyle{myheadings}
\thispagestyle{plain}
\markboth{YUNFENG CAI AND CHENGYU LIU}{NON-ORTHOGONAL GENERAL JOINT BLOCK DIAGONALIZATION}

\section{Introduction}\label{sec:intro}
The joint block diagonalization problem, also called the simultaneous block diagonalization problem, is a particular case of the block term decomposition (BTD) of a third order tensor \cite{de2008decompositions, de2008decompositions2, de2008decompositions3, nion2011tensor}.
Such problem has found many applications in
independent subspace analysis (e.g., \cite{cardoso1998multidimensional, de2000fetal,  theis2005blind, theis2006towards}) and semidefinite programming (e.g., \cite{gatermann2004symmetry, de2007reduction, bai2009exploiting, de2010exploiting}).
To specify the problem, 
we name the problem by nine capital letters {\sc wwxyyzzzz}.
The first two letters, {\sc ww}, indicate the type of the matrices in the matrix set,
{\sc sy/he} for real symmetric/complex Hermitian matrices,
{\sc ge} for general matrices.
The second letter, {\sc x}, indicates that the problem is solved in the exact sense or the approximate sense,   {\sc e} for the former and {\sc a} for the latter.
The next two letters, {\sc yy}, indicate the type of the diagonalizer, {\sc no/nu} for non-orthogonal/non-unitary matrix, {\sc o/u} for orthogonal/unitary matrix(often left as blank).
The last four letters, {\sc zzzz}, indicate the computation performed,
{\sc jd} for joint diagonalization, {\sc jbd} for joint block diagonalization,
{\sc gjbd} for general joint block diagonalization.
Next, we first give some definitions, then formulate the {\sc wweyyjbd} problem and the {\sc wweyygjbd} problem mathematically.
\begin{definition}
We call $\tau_n=(n_1,\dots,n_t)$ a {\em partition} of positive integer $n$ if
$n_1,n_2,\dots,n_t$ are all positive integers and the sum of them is $n$, i.e., $\sum_{i=1}^t n_i=n$.
The integer $t$ is called the {\em cardinality} of the partition $\tau_n$,
denoted by $\card(\tau_n)$.
The set of all partitions of $n$ is denoted by $\mathbb{T}_n$.
\end{definition}

\begin{definition}
Given a partition $\tau_n=(n_1,\dots,n_t)$,
for any matrix $A$ of order $n$,
define its block diagonal part and off-block-diagonal part associated with $\tau_n$ as
  \begin{align*}
    \Bdiag_{\tau_n}(A)=\diag(A_{11},\dots,A_{tt}),\quad   \OffBdiag_{\tau_n}(A)=A-\Bdiag_{\tau_n}(A),
\end{align*}
respectively, where $A_{ii}$ is of order $n_i$ for $i=1,\dots, t$.
A matrix $A$ is referred to as a $\tau_n$-block diagonal matrix if $\OffBdiag_{\tau_n}(A)=0$.
\end{definition}

Let $\sn$, $\hn$, $\on$, $\un$, $\grn$ and $\gcn$
denote the $n\times n$ matrix set of real symmetric matrix, complex Hermitian matrix, real orthogonal matrix,
complex unitary matrix, real nonsingular matrix and complex nonsingular matrix, respectively. 
Let $\an=\sn$, $\hn$, $\rnn$, or $\cnn$,
$\wn=\on$, $\un$, $\grn$, or $\gcn$.
Then the {\sc wweyyjbd} problem and the {\sc wweyygjbd} problem can be stated as:

\noindent\textbf{The \textsc{wweyyjbd} problem.} \quad Given a matrix set $\{A_i\}_{i=1}^m$
with $A_i\in\an$,
and a partition $\tau_n=(n_1,\dots,n_t)$.
Find a matrix $W=W(\tau_n)\in\wn$ such that $W^{\star} A_i W$ for $i=1,\dots, m$
are all $\tau_n$-block diagonal matrices, i.e.,
\begin{equation}\label{eq:nojbd}
W^{\star} A_i W=\diag(A_{i}^{(11)},\dots, A_{i}^{(tt)}), \quad \mbox{for}\quad i=1,2,\dots, m,
\end{equation}
where $A_{i}^{(jj)}\in\mathbb{R}^{n_j\times n_j}$ for $j=1,2,\dots,t$.
Here the symbol $(\cdot)^{\star}$ stands for the transpose of a real matrix or the conjugate transpose of a complex matrix.

\medskip

\noindent\textbf{The  \textsc{wweyygjbd} problem.} \quad
Given a matrix set $\mathcal{A}=\{A_i\}_{i=1}^m$
with $A_i\in\an$.
Find a partition $\tau_n'=(n_1',\dots,n_t')$ and a matrix $W=W(\tau_n')\in\wn$
such that
\begin{equation*}
\card(\tau_n')=\max\{\card(\tau_n) \, \big|\,
\mbox{there exists a $W=W(\tau_n)$ which solves {\sc wweyyjbd}.}\}
\end{equation*}

In practice, the matrices $A_i$'s are usually constructed from empirical data, 
the {\sc wweyyjbd} problem in general has no solutions.
Consequently, the {\sc wwayyjbd}  problem is considered instead.
Naturally, the {\sc wwayyjbd}  problem is formulated as an optimization problem $C(W)=\min$, where $C(\cdot)$ is certain cost function, $W$ belongs to 
certain feasible set, say $\wn$.
In current literature, there are mainly three cost functions for the {\sc wwayyjbd} problem \cite{tichavsky2012algorithms}, namely, $C_{LS}(W)$ \cite{fevotte2007pivot}, $C_{LL}(W)$ \cite{lahat2012joint}, $C_{FIT}(Y)$($Y=W^{-1}$) \cite{nion2011tensor}, which can be respectively given by
\begin{align*}
C_{LS}(W)&=\frac{1}{2} \sum_{i=1}^m \|\OffBdiag_{\tau_n}(W^{\star}A_iW)\|_F^2,\\
C_{LL}(W)&= \frac{1}{2} \sum_{i=1}^m \log \frac{\det(\OffBdiag_{\tau_n}(W^{\star}A_iW))}{\det(W^{\star}A_iW)},\\
C_{FIT}(Y)&=\sum_{i=1}^m\|A_i-Y^{\star}D_iY\|_F^2,
\end{align*}
where $\tau_n\in\tn$ is a prescribed partition, 
$D_i=\argmin_{\OffBdiag_{\tau_n}(D)=0} \|A_i-Y^{\star}DY\|_F^2$ in $C_{FIT}$.

\medskip

Great efforts has been devoted to solving the {\sc wwayyjbd} problem and numerous algorithms are proposed.
For example,  
the JBD-OG/ORG  method by H. Ghennioui et al. \cite{ghennioui2010gradient}, 
the JBD-LM  method by O.~Cherrak et al. \cite{cherrak2013non},
the JBD-NCG method by D. Nion  \cite{nion2011tensor}.
For more methods, 
we refer the readers to \cite{de2009survey, chabriel2014joint, tichavsky2014non} and reference therein.
A very useful {\sc matlab} toolbox for tensor computation -- {\sc tensorlab} \cite{tensorlab}, 
which is available at http://www.tensorlab.net,
is also recommended for interested readers.

The {\sc wwayygjbd} problem, on the other hand,
attempts to maximize $\card(\tau_n)$ and minimize $C(W)$ at the same time,
which results in a rather difficult optimization problem.
In this paper, we formulate the {\sc wwayygjbd} problem as follows:

\medskip

\noindent\textbf{The \textsc{wwayygjbd} problem.} \quad
Given a matrix set $\mathcal{A}=\{A_i\}_{i=1}^m$
with $A_i\in\an$.
Find a partition $\tau_n'=(n_1',\dots,n_t')$ and a nonsingular matrix $W=W(\tau_n')\in\wn$
such that $(\tau_n', W)$ solves the following constrained optimization problem:
\begin{subequations}
\begin{align}
&\max_{\tau_n\in\mathbb{T}_n} & &\card(\tau_n)\label{opt}\\
&\mbox{subject to} & & \sum_{i=1}^m  \|\OffBdiag_{\tau_n}(W^{\star} A_i W)\|_F^2 \le \epsilon^2,\label{fcon}\\
&&&  \Bdiag_{\tau_n}(W^{\star}W)=I_n,\label{wtau}
\end{align}
\end{subequations}
where $\epsilon\ge 0$ is a prescribed parameter.

The first constraint \eqref{fcon} is used to control the norm of the off-block-diagonal parts of $W^{\star}A_i W$'s, where $\epsilon$ is a parameter.
In particular, if $\epsilon=0$, the {\sc wwayygjbd} problem becomes the {\sc wweyygjbd} problem.
The second constraint is used to prevent $W$  from becoming too small.
Notice that if $(\tau_n', W)$ solves the above {\sc wwayygjbd} problem, then so does $(\tau_n', WQ)$,
where $Q\in\on$ or $\un$ is a $\tau_n$-block diagonal matrix;
furthermore, the value on the left hand side of the first constraint remains unchanged for any $Q$.

The {\sc wwxyygjbd} problem is not well studied, in both theory and algorithm.
In current literature,  the {\sc wwayygjbd} problem, as an optimization problem,
is solved by a two stage procedure, in the first stage, apply a {\sc jd} algorithm;
in the second stage, reveal the block structure by permutation.
Such an approach is based on a conjecture \cite{abed2004algorithms} and is only partially proved \cite{theis2006towards}.
From a matrix $\ast$-algebraic point of view, the {\sc gexogjbd}/{\sc gexugjbd} problem
is studied in \cite{maehara2011algorithm,de2011numerical,murota2010numerical,maehara2010numerical}:
based on some structure theorem of the matrix $\ast$-algebra, 
the {\sc geeogjbd}/{\sc geeugjbd} problem is solved in theory, 
and several algorithms are proposed to solve the {\sc gexogjbd}/{\sc gexugjbd} problem.
From a matrix polynomial spectral approach, the {\sc hexnugjbd} problem is discussed in \cite{cai2015matrix}: based on the spectral decomposition of a matrix polynomial, 
the necessary and sufficient condition for the existence of nontrivial solutions are established,
the equivalence of the solutions are given,
and two algorithms are developed to solve the  {\sc hexnugjbd}.
Both the matrix $\ast$-algebra approach and the matrix polynomial approach are from algebraic point of view, and the numerical methods proposed are direct methods rather than iterative methods as in the optimization approach. 

In this paper, we study the {\sc gexnogjbd} problem 
\footnote{The {\sc gexnugjbd} problem can be solved in a similar way.}
via an algebraic approach, similar as in \cite{maehara2011algorithm}.
The tools we employed are only some fundamental matrix decompositions,
rather than the fancy structure theorem of matrix $\ast$-algebra.
Our contributions are fourfold.
First, we show that the existence of the solutions to the {\sc geenogjbd} problem 
is strongly connected with the  null space
\begin{align}\label{na}
\na:=\{Z\in\mathbb{R}^{n\times n}\; | \;
A_iZ=Z^TA_i, \; \mbox{for } i=1,\dots, m\}.
\end{align}
A solution to the {\sc geenogjbd} problem 
can be obtained from a matrix decomposition of a ``generic'' element in $\na$.
Second, a necessary and sufficient condition for the equivalence of the solutions are established.
Third, we show that the solutions to the {\sc geanogjbd} problem can be obtained from 
a decomposition of a ``generic'' element in the ``near-null'' space
\begin{align}\label{approxna}
\{Z\in\mathbb{R}^{n\times n}\; |\;A_iZ\approx Z^TA_i, \; \mbox{for } i=1,\dots,m \}.
\end{align}
Last,  two algorithms are proposed and numerical examples show their merits.

The rest of this paper is organized as follows.
In section~\ref{exact}, we give the existence and equivalence of the solutions to the {\sc geenogjbd} problem, and also show how to determine a solution.
In section~\ref{approx}, we show that the {\sc geanogjbd} problem can be solved in a similar way as the {\sc geenogjbd} problem, two numerical methods are proposed.
The numerical examples are given in  section~\ref{sec:numer}.
Finally, some concluding remarks are given in section~\ref{sec:conclusion}.

{\bf Notation.} The symbol $\otimes$ denotes the Kronecker product. 
The operation $\myvec(X)$ denotes the vectorization of the matrix $X$ formed by stacking the columns of $X$ into a single column vector.
The operation $\reshape(x,m,n)$ is to reshape the $mn$-by-1 vector $x$ into a $m$-by-$n$ matrix, e.g., $\reshape(\myvec(X),m,n)=X$.
The 2-norm, Frobinius norm and infinity norm of a matrix is denoted by $\|\cdot\|_2$, $\|\cdot\|_F$ and $\|\cdot\|_{\infty}$, respectively.
The eigenvalue set of a square matrix $A$ is denoted by $\lambda(A)$.
Let $a$ be a row vector of order $t$,
for convenience, by setting $a_j=(b_1,b_2)$,
we mean that the $j$th element of $a$ is replaced by $(b_1,b_2)$,
that is, $a=(a_1,\dots,a_{j-1},b_1,b_2,a_{j+1},\dots,a_t)$.
We shall also adopt MATLAB convention to access the entries of vectors and matrices. 
The set of integers from $i$ to $j$ inclusive is $i : j$. 
For a matrix $A$, its submatrices $A(k:\ell,i:j)$, $A(k:\ell,:)$, $A(:,i:j)$ consist of intersections of 
row $k$ to row $\ell$ and column $i$ to column $j$, 
row $k$ to row $\ell$ and all columns,
all rows and column $i$ to column $j$, respectively.

\section{On the \textsc{geenogjbd} problem}\label{exact}
In this section, we first discuss the existence of the solutions to {\sc geenogjbd} problem,
then the equivalence of the solutions, and finally show how to determine a solution.

For the ease of our following discussions, we need the following definitions.
\begin{definition}
Let $\tau_n$, $\tilde{\tau}_n\in\tn$ with $\card(\tau_n)=\card(\tilde{\tau}_n)=s$. We say that $\tau_n$ is {\em equivalent} to $\tilde{\tau}_n$
if there exists a permutation $\Pi_s$ such that $\tau_n=\tilde{\tau}_n\Pi_s$,
denoted by $\tau_n\sim \tilde{\tau}_n$.
\end{definition}

\begin{definition}
For any real matrix $Z$ of order $n$, denote its distinct eigenvalues by
$\lambda_1,\bar{\lambda}_1,\dots,\lambda_{\ell},\bar{\lambda}_{\ell},\lambda_{\ell+1},\dots,\lambda_{t}$,
where $\lambda_1,\dots,\lambda_{\ell}$ are non-real, $\lambda_{\ell+1},\dots,\lambda_t$ are real.
Let the algebraic multiplicity of $\lambda_i$ be $m_i$ for $i=1,\dots, t$.
Split the distinct eigenvalues into $s$ non-intersect subsets with each subset closed under complex conjugation,
and denote $n_i$ the sum of the algebraic multiplicities of the eigenvalues in each subset,
then an {\em eigenvalue partition} of $Z$ is defined as $\zeta_n(Z)=(n_1,\dots,n_s)$.
In particular, if $s=t$, such eigenvalue partition, hereafter called {\em optimal eigenvalue partition} 
 and denoted by $\zeta_n^{\opt}(Z)$, 
is unique up to a permutation $\Pi_t$, i.e.,
\[
\zeta_n^{\opt}(Z)=(n_1,\dots,n_t)\sim(2m_1,\dots, 2m_{\ell}, m_{\ell+1},\dots,m_t).
\]
The {\em eigenvalue decomposition} of $Z$ corresponding with a partition $\zeta_n(Z)$ is defined as
\begin{align}\label{eigdecomp}
Z=WGW^{-1}=W\diag(G_1,\dots,G_s)W^{-1},
\end{align}
where $G_j\in\mathbb{R}^{n_j\times n_j}$ for $j=1,\dots,s$,
$\lambda(G_j)\cap\lambda(G_k)=\emptyset$ for $j\ne k$,
$W\in\mathbb{R}^{n\times n}$ is nonsingular.
\end{definition}

\subsection{Existence of the solutions to the \textsc{geenogjbd} problem}
In this subsection, we establish the necessary and sufficient condition for the existence of solutions to 
the {\sc geenojbd} problem and {\sc geenogjbd} problem.

\begin{lemma}\label{lem1}
Given a partition $\tau_n=(n_1,\dots, n_s)\in\mathbb{T}_n$.
Then the {\sc geenojbd} problem has a solution $W$ if and only if there exists a matrix $Z\in\na$ such that
an eigenvalue partition of $Z$, denoted by $\zeta_n(Z)$, is equivalent to $\tau_n$.
\end{lemma}

\begin{proof}
$(\Rightarrow)$ (Sufficiency) If $W$ is a solution to the {\sc geenojbd} problem,
then \eqref{eq:nojbd} holds for $i=1,\dots,m$.
Let
\[
Z=W\diag(G_1,G_2,\dots,G_s)W^{-1}
=W\diag(I_{n_1},2I_{n_2}\dots, s I_{n_s})W^{-1},
\]
 it is easy to see that $Z\in\na$ and
$\zeta_n^{\opt}(Z)$, which is an eigenvalue partition of $Z$, is equivalent to $\tau_n$.

$(\Leftarrow)$ (Necessity) Using the assumption that $Z$ has an eigenvalue partition $\zeta_n(Z)$
that is equivalent to $\tau_n$,  we know that $Z$ has the following eigenvalue decomposition 
corresponding with $\tau_n$:
\begin{align}\label{eq:zwgw}
Z=W\diag(G_1,\dots,G_s)W^{-1},
\end{align}
where $G_j\in\mathbb{R}^{n_j\times n_j}$ for $j=1,\dots,s$,
and $\lambda(G_j)\cap\lambda(G_k)=\emptyset$ for $j\ne k$.
Substituting \eqref{eq:zwgw} into $A_iZ=Z^T A_i$, we get
\begin{align}\label{eq:vgva}
A_i {W}\diag(G_{1},\dots, G_s){W}^{-1}={W}^{- T}\diag(G_{1}^T,\dots, G_s^T){W}^T A_i.
\end{align}
Partition $W^TA_iW=[A_i^{(jk)}]$ with $A_i^{(jk)}\in\mathbb{R}^{n_j\times n_k}$, then
it follows from \eqref{eq:vgva} that
\begin{align}
A_i^{(jk)}G_k=G_j^T A_i^{(jk)}, \qquad \mbox{for } \;  i=1,2,\dots, m, \quad j,k=1,2,\dots,s.
\end{align}
Consequently, for $j\ne k$, we know that $A_i^{(jk)}=0$ since $\lambda(G_j)\cap\lambda(G_k)=\emptyset$.
The conclusion follows.
\qquad\end{proof}

\begin{theorem}\label{thm1}
The {\sc geenogjbd} problem has a solution  $(\tau_n, W)$ if and only if
there exists a matrix $Z\in\na$ which has an eigenvalue partition $\zeta_n(Z)$
that is equivalent to $\tau_n$,
and there is no $\widetilde{Z}\in\na$ such that $\card(\zeta_n^{\opt}({\widetilde{Z}}))>\card(\tau_n)$.
\end{theorem}

\begin{proof}
$(\Rightarrow)$ (Sufficiency) If $(\tau_n,W)$ is a solution to the {\sc geenogjbd} problem,
then $W$ is a solution to the {\sc geenojbd} problem.
By Lemma~\ref{lem1}, we know that
there exists a matrix $Z\in\na$
which has an eigenvalue partition $\zeta_n(Z)$ that is equivalent to $\tau_n$.

If there exists a $\widetilde{Z}\in\na$ such that $\card(\zeta_n^{\opt}(\widetilde{Z}))>\card(\tau_n)$,
then $\widetilde{Z}$ has an eigenvalue decomposition
\begin{align}\label{eq:zwgw2}
\widetilde{Z}=\widetilde{W}\diag(\widetilde{G}_1,\dots,\widetilde{G}_{\tilde{t}})\widetilde{W}^{-1},
\end{align}
where $\tilde{t}>\card(\tau_n)$, $\widetilde{G}_j$ is of order $\tilde{n}_j$,
and $\lambda(\widetilde{G}_j)\cap\lambda(\widetilde{G}_k)=\emptyset$.
By Lemma~\ref{lem1}, for the partition $\tilde{\tau}_n=(\tilde{n}_1,\dots,\tilde{n}_{\tilde{t}})$,
the {\sc geenojbd} problem has a solution $\widetilde{W}$.
Therefore, for any solution of the {\sc geenogjbd} problem,
the cardinality of the partition should be no less than $\tilde{t}$,
which contradicts with the fact that $(\tau_n, W)$ is a solution to the {\sc geenogjbd} problem
and $\card(\tau_n)<\tilde{t}$.

$(\Leftarrow)$ (Necessity)
If there exists a matrix $Z\in\na$ which has an eigenvalue partition $\zeta_n(Z)$
that is equivalent to $\tau_n$,
then by Lemma~\ref{lem1},
for the partition $\tau_n$, the {\sc geenojbd} problem has a solution $W$.
If there is no $\widetilde{Z}\in\na$ such that $\card(\zeta_n^{\opt}({\widetilde{Z}}))>\card(\tau_n)$,
we declare that $(\tau_n,W)$ is a solution to the {\sc geenogjbd} problem.
Because otherwise, let $(\tilde{\tau}_n,\widetilde{W})$ be a solution to the {\sc geenogjbd} problem, then $\card(\tilde{\tau}_n)>\card(\tau_n)$.
In another word, for the partition $\tilde{\tau}_n$, the {\sc geenojbd} problem has a solution $\widetilde{W}$.
By Lemma~\ref{lem1}, there exists a matrix $\widetilde{Z}\in\na$
such that \eqref{eq:zwgw2} holds.
Then it follows that $\card(\zeta_n^{\opt}(\widetilde{Z}))\ge \tilde{t}=\card(\tilde{\tau}_n)>\card(\tau_n)$,
which is a contradiction.
\qquad\end{proof}

\begin{remark}\label{rem2}
Given a matrix $Z\in\na$,
if $Z$ has an eigenvalue partition $\zeta_n(Z)=(n_1,\dots, n_s)$,
it has an eigenvalue decomposition corresponding with $\zeta_n(Z)$, i.e.,
$Z=W\diag(G_1,\dots, G_s)W^{-1}$ with $G_j\in\mathbb{R}^{n_j\times n_j}$
and $\lambda(G_j)\cap\lambda(G_k)=\emptyset$.
Then for the partition $(n_1,\dots, n_s)$, the {\sc geenojbd} problem has a solution $W$.
Notice that $s\le\card(\zeta_n^{\opt}(Z))$, and the equality holds if and only if
the eigenvalues of each $G_j$ are the same real number or a complex conjugate pair.
Therefore, if $(\zeta_n(Z),W)$ is a solution to the {\sc geenogjbd} problem,
then $\zeta_n(Z)\sim\zeta_n^{\opt}(Z)$,
which implies that the eigenvalues of each $G_j$ are the same real number or a complex conjugate pair.
\end{remark}

\begin{remark}\label{rem3}
If the  {\sc geenogjbd} problem  of matrix set $\mathcal{A}_+=\{I_n\}\cup\mathcal{A}$ 
has a solution $(\tau_n, W)$,
then $(\tau_n, W(W^TW)^{-\frac{1}{2}})$ is a solution to 
the {\sc geeogjbd} problem of matrix set $\mathcal{A}$.
Therefore, the solutions to the {\sc geeogjbd} problem of $\mathcal{A}$ can be obtained by solving the {\sc geenogjbd} problem of $\mathcal{A}_+$.
To be specific, the null space of the matrix set $\mathcal{A}_+$ can be given by
\begin{align*}
\mathscr{N}({\mathcal{A}_+})
&=\{Z\in\mathbb{R}^{n\times n}\; | \; Z^T=Z,  A_iZ=Z^TA_i, \; \mbox{for } i=1,\dots, m\}\\
&=\{Z\in\mathbb{R}^{n\times n}\; | \; Z^T=Z, A_iZ=ZA_i, \; \mbox{for } i=1,\dots, m\},
\end{align*}
which is the {\em commutant algebra} of the matrix $\ast$-subalgebra generated by $\mathcal{A}_+$ \cite{maehara2011algorithm}.
For a generic $Z\in\mathscr{N}({\mathcal{A}_+})$,  
let $Z=QTQ^T$ be
its spectral decomposition (also its eigenvalue decomposition corresponding with $\zeta_n^{\opt}(Z)$),
where $Q$ is orthogonal, $T=\diag(\lambda_1 I_{n_1},\dots, \lambda_tI_{n_t})$
with $\lambda_i\ne\lambda_j$ for $i\ne j$.
Then according to Proposition~3.1 in \cite{maehara2011algorithm},
$(\tau_n, Q)$ solves the {\sc geeogjbd} problem of $\mathcal{A}$,
which agrees with Theorem~\ref{thm1} here.
\end{remark}

\subsection{Equivalence of the solutions}

If $(\tau_n,W)$ with $\tau_n = (n_1, \dots, n_t)$ is a solution to the {\sc geenogjbd} problem,
then so is $(\hat{\tau}_n,\widehat{W})=(\tau_n\Pi_t, WT\Pi)$,
where $\Pi_t$ is a permutation matrix of order $t$,
$\Pi\in\mathbb{R}^{n\times n}$ is permutation matrix,
which can be obtained by replacing the 1 and 0 elements in $j$th row of $\Pi_t$ by
$I_{n_j}$ and zero matrices of right sizes, respectively 
(hereafter such permutation matrix $\Pi$ is referred to as the {\em block permutation matrix corresponding with $\tau_n$}),
and $T = \diag(T_{jj})$ is a nonsingular $\tau_n$-block diagonal matrix.
We write $(\tau_n,W)\sim(\hat{\tau}_n,\widehat{W})$ if $(\hat{\tau}_n,\widehat{W})=(\tau_n\Pi_t, WT\Pi)$.
Notice that the relation $\sim$ is  reflexive, symmetric, and transitive,
i.e., it is an equivalence relation.
Consequently, we may say that $(\tau_n,W)$ and $(\hat{\tau}_n,\widehat{W})$
are {\em equivalent}.
A fundamental problem is: are all solutions to the {\sc geenogjbd} problem equivalent?
The following theorem gives the answer.

\begin{theorem}\label{thm:uniq}
Suppose that the {\sc geenogjbd} problem has a solution $(\tau_n, W)$,
where $\tau_n=(n_1,\dots, n_t)\in\mathbb{T}_n$, $W$ satisfies \eqref{eq:nojbd}.
For $j=1,\dots,t$, let $\mathcal{A}_j=\{A_i^{(jj)}\}_{i=1}^m$.
The following statements are equivalent:
\begin{enumerate}
\item[(1)]
All solutions to the {\sc geenogjbd} problem are equivalent.
\item[(2)]
The dimension of $\na$ equals to the sum of the dimension of
$\mathscr{N}(\mathcal{A}_j)$, i.e.,
\begin{align}\label{dimna}
\dim\na=\sum_{j=1}^t\dim\mathscr{N}(\mathcal{A}_j).
\end{align}
\item[(3)]
For any $1\le j< k \le t$, the matrix
\begin{align}\label{mjk}
M_{jk}=\sum_{i=1}^t\bsmat I_{n_k}\otimes[A_i^{(jj)^T} A_i^{(jj)} +A_i^{(jj)} A_i^{(jj)^T}] &
A_i^{(kk)}\otimes A_i^{(jj)} +A_i^{(kk)^T}\otimes A_i^{(jj)^T}\\
A_i^{(kk)}\otimes A_i^{(jj)} +A_i^{(kk)^T}\otimes A_i^{(jj)^T} &
[A_i^{(kk)^T} A_i^{(kk)} +A_i^{(kk)} A_i^{(kk)^T}] \otimes I_{n_j}
\esmat
\end{align}
is nonsingular.
\end{enumerate}
\end{theorem}

\begin{proof}
We proceed by showing $(1)\Rightarrow(2)\Rightarrow(1)$ and $(2)\Leftrightarrow(3)$.

$(1)\Rightarrow(2)$
Notice that for any $F_j\in\mathscr{N}(\mathcal{A}_j)$, $j=1,\dots,t$,
it is easy to see that $W \diag(F_1,\dots, F_t)W^{-1}\in\na$.
Therefore, $\dim\na\ge\sum_{j=1}^t\dim\mathscr{N}(\mathcal{A}_i)$.
Next, we show $\dim\na\le\sum_{j=1}^t\dim\mathscr{N}(\mathcal{A}_i)$
by showing that for any $\widehat{Z}\in\na$, it can be written in the form
$\widehat{Z}=W \diag(F_1,\dots, F_t)W^{-1}$ with $F_j\in\mathscr{N}(\mathcal{A}_j)$ for $j=1,\dots,t$.

Now let $Z=W\diag(G_1,\dots,G_t)W^{-1}$,
$Z_{\epsilon}=Z+\epsilon \widehat{Z}$,
where $G_j\in\mathscr{N}(\mathcal{A}_j)$ for $j=1,\dots,t$, $\lambda(G_j)\cap\lambda(G_k)=\emptyset$ for $j\ne k$,
and $\epsilon$ is a positive parameter.
On one hand, for a sufficient small $\epsilon$,
we know that the number of distinct eigenvalues of $Z_{\epsilon}$ should be no less than
that of $Z$, i.e., $\card(\zeta_n^{\opt}(Z_{\epsilon}))\ge \card(\zeta_n^{\opt}(Z))\ge t$;
on the other hand, notice that ${Z}_{\epsilon}\in\na$ 
since $Z$, $\widehat{Z}\in\na$.
Then by Theorem~\ref{thm1}, $\card(\zeta_n^{\opt}({Z}_{\epsilon}))\le \card(\tau_n)=t$.
Then it follows that $\card(\zeta_n^{\opt}(Z_{\epsilon}))=\card(\zeta_n^{\opt}(Z))=t$.
For the partition $\zeta_n^{\opt}(Z_{\epsilon})$, by Lemma~\ref{lem1},
the {\sc geenojbd} problem has a solution $W_{\epsilon}$ satisfying
\begin{align}\label{ze1}
Z_{\epsilon}=W_{\epsilon}\diag(G_{1\epsilon},\dots,G_{t\epsilon})W_{\epsilon}^{-1},
\end{align}
where 
$\lambda(G_{j\epsilon})\cap\lambda(G_{k\epsilon})=\emptyset$ for $j\ne k$.
Now that $(\zeta_n^{\opt}(Z_{\epsilon}),W_{\epsilon})$ is also a solution to the {\sc geenogjbd} problem,
and hence it is equivalent to $(\tau_n,W)$, i.e.,
there exists a permutation matrix $\Pi_t\in\mathbb{R}^{t\times t}$ and
a nonsingular $\tau_n$-block diagonal matrix $T=\diag(T_{jj})$ such that
\begin{align}\label{zetaze}
\zeta_n^{\opt}(Z_{\epsilon})=\tau_n\Pi_t,\quad W_{\epsilon}=WT\Pi,
\end{align}
where $\Pi$ is the block  permutation matrix corresponding with $\tau_n$.
Using \eqref{ze1} and \eqref{zetaze}, we get
\begin{align*}
Z_{\epsilon}=WT\Pi\diag(G_{1\epsilon},\dots,G_{t\epsilon})\Pi^T T^{-1}W^{-1}
=W\diag(F_{1\epsilon},\dots,F_{t\epsilon})W^{-1},
\end{align*}
where $F_{j\epsilon}=T_{jj}G_{i_j\epsilon}T_{jj}^{-1}\in\mathbb{R}^{n_j\times n_j}$
for $j=1,\dots,t$,
$\{i_1,\dots,i_t\}$ is a permutation of $\{1,\dots, t\}$.
Using $Z_{\epsilon}\in\na$,
it is easy to see that $F_{j\epsilon}\in\mathscr{N}(\mathcal{A}_j)$.
Therefore, $F_{j\epsilon}-G_j\in\mathscr{N}(\mathcal{A}_j)$ for $j=1,\dots,t$.
Then it follows that
\[
\widehat{Z}=\frac{1}{\epsilon}(Z_{\epsilon}-Z)
=\frac{1}{\epsilon}W\diag(F_{1\epsilon}-G_1,\dots,F_{t\epsilon}-G_t)W^{-1},
\]
which is the required form.

$(2)\Rightarrow(1)$
Let $(\hat{\tau}_n,\widehat{W})$ be a solution to the {\sc geenogjbd} problem.
It suffices if we can show that $(\hat{\tau}_n,\widehat{W})$ and $(\tau_n,W)$ are equivalent.

First, as both $(\hat{\tau}_n,\widehat{W})$ and $(\tau_n,W)$ are solutions to the {\sc geenogjbd} problem, we know that $\card(\hat{\tau}_n)=\card(\tau_n)=t$.
Let $\hat{\tau}_n=(\hat{n}_1,\dots,\hat{n}_t)$.
By Theorem~\ref{thm1}, there exists a $\widehat{Z}\in\na$
such that
\begin{align}\label{eq:zwgw3}
\widehat{Z}=\widehat{W}\diag(\widehat{G}_1,\dots,\widehat{G}_t)\widehat{W}^{-1},
\end{align}
where $\widehat{G}_j$ is of order $\hat{n}_j$ and $\lambda(\widehat{G}_j)\cap\lambda(\widehat{G}_k)=\emptyset$ for $j\ne k$.
Second, the equality \eqref{dimna} implies that for any element in
$\na$, in particular $\widehat{Z}$, there exist $\widehat{F}_j\in\mathscr{N}(\mathcal{A}_j)$ for $j=1,\dots,t$
such that
\begin{align}\label{eq:zwfw}
\widehat{Z}=W \diag(\widehat{F}_1,\dots, \widehat{F}_t)W^{-1}.
\end{align}
Combining \eqref{eq:zwgw3} and \eqref{eq:zwfw}, we have
\begin{align}\label{wgwwfw}
\widehat{W}\diag(\widehat{G}_1,\dots,\widehat{G}_t)\widehat{W}^{-1}
=W \diag(\widehat{F}_1,\dots, \widehat{F}_t)W^{-1}.
\end{align}
Noticing that the eigenvalues of $\widehat{G}_j$ are the same real number or the same complex conjugate pair,
and so are the eigenvalues of $\widehat{F}_j$,
we know that there is a permutation matrix $\Pi_t\in\mathbb{R}^{t\times t}$ such that
$\hat{\tau}_n=\tau_n\Pi_t$.
The corresponding block permutation matrix $\Pi$ satisfies
$\Pi\diag(\widehat{G}_1,\dots,\widehat{G}_t)\Pi^T
=\diag(\widehat{G}_{i_1},\dots,\widehat{G}_{i_t})$,
where $\{i_1,\dots,i_t\}$ is a permutation of $\{1,\dots, t\}$,
and  for $j=1,\dots,t$, $\widehat{G}_{i_j}$ is similar to $\widehat{F}_j$,
i.e., there exists a nonsingular matrix
$T_{jj}$ of order $n_j$ satisfying $\widehat{F}_j=T_{jj}\widehat{G}_{i_j}T_{jj}^{-1}$.
Then \eqref{wgwwfw} can be rewritten as
\[
\widehat{W}\Pi^T \diag(\widehat{G}_{i_1},\dots,\widehat{G}_{i_t})\Pi\widehat{W}^{-1}
=WT_1\diag(\widehat{G}_{i_1},\dots,\widehat{G}_{i_t})T_1^{-1}W^{-1},
\]
where $T_1=\diag(T_{11},\dots,T_{tt})$ is a nonsingular $\tau_n$-block diagonal matrix.
Using $\lambda(G_j)\cap\lambda(G_k)=\emptyset$ for $j\ne k$, we know that
$T_1^{-1}W^{-1}\widehat{W}\Pi^T$ is a $\tau_n$-block diagonal matrix, denoted by $T_2$.
Thus, $\widehat{W}=W T_1T_2\Pi$.
The conclusion follows immediately since $T_1T_2$ is nonsingular and $\tau_n$-block diagonal.

$(2)\Leftrightarrow(3)$ For any $Z\in\na$, partition $W^{-1}ZW$ as $W^{-1}ZW=[Z_{jk}]$ with $Z_{jk}\in\mathbb{R}^{n_j\times n_k}$.
Then we have the following equivalence relations, which completes the proof.

$(2)\Leftrightarrow$ For any  $j\ne k$, $Z_{jk}=0$.\\
$\Leftrightarrow$ For any $j< k$, the solution to matrix equations
\begin{align*}
\bsmat A_i^{(jj)}& 0\\ 0& A_i^{(kk)}\esmat
\bsmat Z_{jj} & Z_{jk} \\ Z_{kj} & Z_{kk}\esmat
= \bsmat Z_{jj} & Z_{jk} \\ Z_{kj} & Z_{kk}\esmat^T
\bsmat A_i^{(jj)}& 0\\ 0& A_i^{(kk)}\esmat, \quad \mbox{for $i=1,\dots,m$}
\end{align*}
must be in a block diagonal form $\diag(Z_{jj},Z_{kk})$.\\
$\Leftrightarrow$ For any $j< k$, the solutions to matrix equations
 \begin{align}\label{azza0}
A_i^{(jj)}Z_{jk}-Z_{kj}^TA_i^{(kk)}=0, \quad A_i^{(kk)}Z_{kj}-Z_{jk}^TA_i^{(jj)}=0,\quad 
\mbox{for  $i=1,\dots,m$}
\end{align} 
must be zero.\\
$\Leftrightarrow$ For any $j< k$, the coefficient matrix $\widehat{M}_{jk}$ of the following linear system of equations is of full column rank:
\[
\bsmat I_{n_k}\otimes A_1^{(jj)} & A_1^{(kk)^T}\otimes I_{n_j}\\
I_{n_k}\otimes A_1^{(jj)^T} & A_1^{(kk)}\otimes I_{n_j}\\
\vdots & \vdots\\
I_{n_k}\otimes A_m^{(jj)} & A_m^{(kk)^T}\otimes I_{n_j}\\
I_{n_k}\otimes A_m^{(jj)^T} & A_m^{(kk)}\otimes I_{n_j}
\esmat
\bsmat \myvec(Z_{jk}) \\ -\myvec(Z_{kj}^T)\esmat=0.
\]
$\Leftrightarrow$ For any $j< k$, $\widehat{M}_{jk}^T\widehat{M}_{jk}=M_{jk}$ is nonsingular.
\qquad\end{proof}

\bigskip

By the ``(2)$\Leftrightarrow $(3)'' part of the proof for Theorem~\ref{thm:uniq},
we know that the equivalence of all solutions is equivalent to 
that \eqref{azza0} has only zero solution.
Notice that \eqref{azza0} has $2n_jn_k$ unknowns, $2m n_jn_k$ equations.
Since $m>1$, the number of equations is more than the number of unknowns.
Thus, for randomly generated $A_i^{(jj)}$ and $A_i^{(kk)}$,
the solution to \eqref{azza0} equals to zero, almost surely.
Therefore, we may say that for any matrix set $\{A_i\}_{i=1}^m$, the solutions to the corresponding {\sc geenogjbd} problem are equivalent, almost surely.
However, not surprisingly, there are indeed artificial examples in which  the solutions are not equivalent.
Perhaps, the simplest example is the {\sc syejd} problem of 
the scalar matrix set $\{\alpha_i I_n\}_{i=1}^m$,
where $\alpha_i$'s are real numbers.
It is obvious that any orthogonal matrix is a solution, but they are not equivalent.
Another nontrivial example can be constructed by finding some special $A_i^{(jj)}$
and $A_i^{(kk)}$ such that \eqref{azza0} has nontrivial solutions.
See the following example.
\begin{example}\label{eg:nonuniq}
Consider the {\sc syenogjbd} problem of a matrix set $\{A_i\}_{i=1}^m$,
where 
\[
A_i=\diag\left(\bsmat 0 & a_i \\ a_i & b_i\esmat, \bsmat 0 & a_i \\ a_i & b_i\esmat\right),
\] 
 $a_i$'s, $b_i$'s are nonzero real numbers.
Let $W_4=\bsmat 1 & 0 & 0 & -1\\ 0 & 1 & 0 & 0\\ 0 & 1 & 1 &0\\ 0 & 0 & 0 & 1\esmat$,
it is easy to check that both $((2,2), I_4)$ and $((2,2), W_4)$ are solutions.
However, $W_4$ and $I_4$ are not equivalent.
\end{example}

\bigskip

Suppose that $(\tau_n, W)$ is a solution to the {\sc geenogjbd} problem of $\{A_i\}_{i=1}^m$ satisfying \eqref{eq:nojbd}.
Let $\mathcal{T}=[t_{jki}]=[A_i^{(jk)}]\in\mathbb{R}^{n\times n\times m}$ be a third order tensor.
Partition $V=W^{-T}=[V_1\dots V_t]$ with $V_j\in\mathbb{R}^{n\times n_j}$ 
for $j=1,\dots, t$.
Then follow the notations in \cite{de2008decompositions2}, in term of tensor decomposition, \eqref{eq:nojbd} is equivalent to the following type-2  BTD (a generation to the decomposition in rank-$(L_r,L_r,\cdot)$ terms):
\begin{align}\label{td2}
\mathcal{T}=\sum_{j=1}^t \mathcal{C}_j  \bullet_1\mathbf{A}_j \bullet_2 \mathbf{B}_j, 
\end{align}
where $\mathcal{C}_j\in\mathbb{R}^{n_j\times n_j\times m}$ with $\mathcal{C}_j(:,:,i)=A_i^{(jj)}$,
$\mathbf{A}_j=\mathbf{B}_j=V_j$. 
The uniqueness of tenor decompositions are of great importance in applications,
and there is a long list of studies on this subject, 
e.g., \cite{kruskal1977three, de2008decompositions2,stegeman2011uniqueness, domanov2013uniqueness, domanov2013uniqueness2, sorensen2015coupled,sorensen2015new}.
To the best of the authors' knowledge, these studies are mainly dedicated to the sufficient conditions under which the tensor decompositions are unique.  
In particular, in \cite{de2008decompositions2} the essential uniqueness of BTDs was discussed by L. De Lathauwer,
where under some mild conditions, it is shown that type-2 BTD is essentially unique (by \cite[Theorem 6.1 and Remark 6]{de2008decompositions2}).
Theorem~\ref{thm:uniq} we present here gives  not only a sufficient condition but also necessary  one, 
for the uniqueness of this particular BTD in \eqref{td2}, 
from which we can also conclude that BTD in \eqref{td2} is essentially unique.
What's more,  the necessary and sufficient condition
enables us to construct examples that BTDs are not unique 
(e.g.,  example~\ref{eg:nonuniq}),
and also provides us a way to check the BTD  is unique once a BTD is available: 
first, for any $Z_j\in\mathscr{N}(\mathcal{A}_j)$, $j=1,\dots,t$,
check that all eigenvalues of $Z_j$ are the same real number or the same complex conjugate pair (see Remark~\ref{rem2});
second,  for $1\le j<k\le t$, check that $M_{jk}$'s defined in \eqref{mjk} are nonsingular.

\bigskip

Under the assumption that all solutions to the {\sc geenogjbd} problem are equivalent,
we can define the partition of the null space $\na$ as 
\begin{align}\label{zetana}
\zeta(\na):=\{\zeta_n^{\opt}(Z)\; | \; Z=\argmax\{\card(\zeta_n^{\opt}(Z)) \; | \;  Z\in\na\}\},
\end{align}
which forms an {\em equivalence class}.
What's more, we have
\begin{theorem}\label{thm:max}
If all solutions to the {\sc geenogjbd} problem are equivalent,
then for almost all  $Z\in\na$,
 $\zeta_n^{\opt}(Z)\in \zeta(\na)$.
\end{theorem}

\begin{proof}
Follow the notations in Theorem~\ref{thm:uniq}.
On one hand, for any $F_j\in\mathscr{N}(\mathcal{A}_j)$, $j=1,\dots,t$,
we know that $\alpha_j F_j \in\mathscr{N}(\mathcal{A}_j)$,
where $\alpha_j$ is any real number.
By Remark~\ref{rem2},
the eigenvalues of $F_j$ are the same real number $\lambda_j$
or a complex conjugate pair $\{\lambda_j,\bar{\lambda}_j\}$.
Then the possibility
$\mathbb{P}_{jk}(\alpha_j\lambda_j=\alpha_k \lambda_k \mbox{ or }\alpha_k\bar{\lambda}_k)=0$ for any $j\ne k$.
On the other hand,
when all solutions to the {\sc geenogjbd} problem are equivalent,
by Theorem~\ref{thm:uniq},
we know that for any $Z\in\na$,
it can be written as $Z=W \diag(\widehat{F}_1,\dots, \widehat{F}_t)W^{-1}$,
where $\widehat{F}_j\in\mathscr{N}(\mathcal{A}_j)$, $j=1,\dots,t$.
Using $\mathbb{P}_{jk}=0$, we know that $\lambda(\widehat{F}_j)\cap\lambda(\widehat{F}_k)=\emptyset$ almost surely.
Therefore, $\card(\zeta_n^{\opt}(Z))$ is maximized almost surely.
By the definition of $\zeta(\na)$, the conclusion follows.
\qquad\end{proof}

\subsection{Determining a solution to the \textsc{geenogjbd} problem}\label{subsec:solvenogjbd}
If all solutions to the {\sc geenogjbd} problem are equivalent,
using Theorems~\ref{thm1} and \ref{thm:max}, theoretically, we can solve the {\sc geenogjbd} problem by the following procedure:
\begin{enumerate}
\item[] Step~1, solve a basis of $\mathcal{N}(\mathcal{A})$, denote by $\{Z_1,Z_2,\dots,Z_{\ell}\}$;
\item[] Step~2, set $Z=\sum_j{\alpha_j Z_j}$, where $\alpha_j$'s are random real numbers;
\item[] Step~3, compute the eigenvalue decomposition of $Z$ corresponding with $\zeta_n^{\opt}(Z)$ as in \eqref{eigdecomp}.
\end{enumerate}
Then $(\zeta_n^{\opt}(Z), W)$ is a solution to the {\sc geenogjbd} problem.

Some details follow.
In step~1, using the Kronecker product notation, $A_i Z=Z^T A_i$ for $i=1,\dots, m$ is equivalent to
\begin{align}
\mathcal{L}\myvec(Z)=0,
\end{align}
where
\begin{align}\label{eq:L}
\mathcal{L}=\bsmat I\otimes A_i - A_i^T\otimes I \Pi\\
\vdots\\
I\otimes A_i - A_i^T\otimes I \Pi
\esmat,
\end{align}
$\Pi$ is the perfect shuffle permutation \cite[Chap. 12.3]{van1996matrix} of order $n^2$ such that $\Pi \myvec(Z)=\myvec(Z^T)$.
By computing the singular value decomposition (SVD) of $\mathcal{L}$,
we can obtain a basis of $\na$ from a basis of $\mathscr{N}(\mathcal{L})$. 
In step 3, we first compute the Schur decomposition of $Z=QTQ^T$ (with proper ordering of the eigenvalues), 
then  $\zeta_n^{\opt}(Z)$ can be determined by the algebraic multiplicities of the distinct eigenvalues of $T$,
 the eigenvalue decomposition of $Z$ can be computed via Algorithm~7.6.3 in \cite{van1996matrix}.
In next section, we will discuss the numerical methods for the {\sc geanogjbd} problem in detail,
which of course can be applied to the {\sc geenogjbd} problem by simply setting $\epsilon=0$.

\section{Solving the \textsc{geanogjbd} problem}\label{approx}
In this section, 
we show that  the {\sc geanogjbd} problem can be solved in a similar procedure
as the procedure in subsection~\ref{subsec:solvenogjbd} for the {\sc geenogjbd} problem:
\begin{enumerate}
\item  find  a ``basis'' for a ``near-null space'' of $\mathcal{A}$ \eqref{approxna};
\item determine a $Z$ from the ``near-null space'';
\item determine an eigenvalue partition of $Z$ and compute its corresponding eigenvalue decomposition.
\end{enumerate}
However, some fundamental questions need to be answered first:
\begin{enumerate}
\item What is a ``near-null space'' precisely?
\item When $Z$ in the ``near-null space'' has an eigenvalue decomposition  \eqref{eigdecomp} with certain eigenvalue partition $\tau_n$,
is the value of the corresponding cost function
\begin{equation}\label{ftauw}
f(\tau_n,W):=\sum_{i=1}^m\|\OffBdiag_{\tau_n}(W^T A_iW)\|_F^2
\end{equation} 
small?
\item How to determine a $Z$ from the ``near-null space''
such that the eigenvalues of $Z$ has as many clusters as possible and the gap between different clusters is as large as possible?
\end{enumerate}
In subsection~\ref{subsec:fun}, we will first answer questions 1 and 2, 
and then discuss some properties of the eigenvalues of $Z$,
which will be used to cluster the eigenvalues.
For question 3,
two ways to determine the matrix $Z$ and the its eigenvalue partition are proposed,
which leads to two numerical methods for the {\sc geanogjbd} problem, 
namely, {\sc geanogjbd-greedy} and {\sc geanogjbd-consv}, 
which are presented in subsections~\ref{subsec:alg1} and \ref{subsec:alg2}, respectively.

\subsection{Some fundamentals}\label{subsec:fun}
Generally speaking, the null space $\na$ for matrices that can not be exactly joint diagonalized, can be spanned by $\{I_n\}$, which give a trivial solution $((n), I_n)$
to the {\sc geanogjbd} problem.
In order to find a nontrivial solution, we need to define a  ``near-null space'' for the matrix set $\mathcal{A}$.
Let the SVD of $\mathcal{L}$ be 
\begin{align}\label{svdl}
\mathcal{L}=\mathcal{U}\Sigma\mathcal{V}^T,
\end{align}
where $\mathcal{L}$ is defined in \eqref{eq:L},
$\mathcal{U}\in\mathbb{R}^{mn^2\times mn^2}$ and $\mathcal{V}=[v_1\dots v_{n^2}]\in\mathbb{R}^{n^2\times n^2}$ are both orthogonal matrices,
the main diagonal elements of $\Sigma$ are
\[
\sigma_1\ge \sigma_2\ge \dots\ge \sigma_{n^2-\ell}>\delta\ge \sigma_{n^2-\ell+1}\ge \dots\ge \sigma_{n^2}\ge 0,
\]
$\delta \ge0$ is a parameter.
We define the {\em $\delta$-null space} of $\mathcal{A}$ as
\begin{equation}\label{deltana}
\nad:=\subspan\{\reshape(v_{n^2-\ell+j},n,n),\, j=1,\dots,\ell\}.
\end{equation}
Then for any $Z\in\nad$, it holds that
$\|\mathcal{L}\myvec(Z)\|_2\le \delta \|\myvec(Z)\|_2$,
which is equivalent to
$\sum_{i=1}^m\|A_iZ-Z^TA_i\|_F^2\le \delta^2 \|Z\|_F^2$.
Hereafter, we take $\nad$ as the  ``near-null space'',
which is controlled by the parameter $\delta$, the larger $\delta$ is , the larger the space is.
Note also that $\mathscr{N}(\mathcal{A};0)=\na$.
This answers question 1.

\medskip

The following theorem gives an answer to question 2.

\begin{theorem}\label{thm:anogjbd}
If $\sum_{i=1}^m\|A_iZ-Z^TA_i\|_F^2\le \delta^2 \|Z\|_F^2$,
and $Z$ has an eigenvalue decomposition \eqref{eigdecomp}  with $W$ satisfying \eqref{wtau},
then
\begin{align}\label{ineq:offblk}
f(\tau_n,W)
\le \frac{\delta^2\|Z\|_F^2\|W\|_2^4}{\sep(G)^2},
\end{align}
where $f(\tau_n,W)$ is defined in \eqref{ftauw}, $\sep(G)=\min_{j\ne k}\sep(G_j^T,G_k)$, and $\sep(G_j^T,G_k)=\min_{X}\frac{\|G_j^T X-XG_k\|_F}{\|X\|_F}$.
\end{theorem}

\begin{proof}
Let $W^TA_iW=[A_i^{(jk)}]$.
Direct calculations give rise to
\begin{align*}
\delta^2\|Z\|_F^2\|W\|_2^4
&\ge \|W^T\|_2^2 \sum_{i=1}^m \|A_iZ-Z^TA_i\|_F^2 \|W\|_2^2\\
&\ge \sum_{i=1}^m \|[A_i^{(jk)}]G-G^T[A_i^{(jk)}]\|_F^2\\
&\ge \sum_{i=1}^m \sum_{1\le j\ne k\le s}\|A_i^{(jk)}G_k-G_j^TA_i^{(jk)}\|_F^2\\
&\ge \sum_{i=1}^m \sum_{1\le j\ne k\le s} \sep(G)^2 \|A_i^{(jk)}\|_F^2\\
&= \sep(G)^2 \sum_{i=1}^m  \|\OffBdiag_{\tau_n}(W^T A_i W)\|_F^2.
\end{align*}
The conclusion follows.
\qquad\end{proof}

If $\delta$ is small and $\sep(G)/\|Z\|_F$ is not, then the right hand side of \eqref{ineq:offblk} will be small,
which means that $f(\tau_n, W)$ is small.
In general
$\sep(G)/\|Z\|_F$ is large if the gap between $\lambda(G_j)$ and $\lambda(G_k)$ is large \cite{xu1997lower}.
Therefore, when solving the {\sc geanogjbd} problem with the procedure in subsection~\ref{subsec:solvenogjbd}, it is critical to 
choose a ``proper'' $\delta$
for the approximate null space $\mathscr{N}(\mathcal{A};\delta)$
and a ``good'' $Z\in\nad$ in the sense that the eigenvalues of $Z$ has as many clusters as possible and the gap between different clusters is as large as possible.  

How to choose a ``proper'' $\delta$ can be very tricky.
A small $\delta$ will lead to a small $f(\tau_n,W)$, but a small $\card(\tau_n)$ too;
A large $\delta$, on the other hand, will lead to a large $\card(\tau_n)$,
but also  a large $f(\tau_n, W)$.
Notice that $I_n\in\na\subset\mathscr{N}(A;\delta)$, 
then the smallest singular value  $\sigma_{n^2}$ of $\mathcal{L}$ 
defined in \eqref{eq:L} must be zero,
since $\myvec(I_n)$ is a right singular vector of $\mathcal{L}$ corresponding with the zero singular value.
The second smallest singular value $\sigma_{n^2-1}$ of $\mathcal{L}$ is in general nonzero for the {\sc geanogjbd} problem.
The similarity transformation matrix obtained from the eigenvalue decomposition of 
$Z$ (reshaped from the right singular vector of $\mathcal{L}$ corresponding with $\sigma_{n^2-1}$)
corresponding with some eigenvalue partition of $Z$, 
is usually a good solution to the {\sc geanojbd} problem of $\mathcal{A}$.
Thus, it is reasonable to set $\delta=\gamma\sigma_{n^2-1}$,
where $\gamma>1$ is some constant.

Finding the ``best'' $Z$ to fully answer question 3 is difficult. 
In next two subsections,
we propose two ways to find the matrix $Z$ and its eigenvalue partition: a greedy way and a conservative way,
which leads to the algorithms {\sc geanogjbd-greedy} and {\sc geanogjbd-consv}, respectively.

In order to determine the eigenvalue partition of $Z\in\nad$, 
what follows we discuss some properties of the eigenvalues of $Z$.

\begin{theorem}\label{rp}
For any $Z\in\nad$, let $(\lambda,x)$ be an eigenpair of $Z$ and $\|x\|_2=1$.
If $ \sum_{i=1}^m  |x^*A_i x|^2\ne 0$, then
\begin{align}
|\lambda-\bar{\lambda}|\le \frac{\delta^2\|Z\|_F^2}{ \sum_{i=1}^m  |x^*A_i x|^2}.
\end{align}
\end{theorem}

\begin{proof}
It follows from $Z\in\nad$ that $\sum_{i=1}^m\|A_iZ-Z^TA_i\|_F^2\le \delta^2\|Z\|_F^2$.
Then using $Zx=\lambda x$, we have
\begin{align*}
\delta^2\|Z\|_F^2
&\ge \sum_{i=1}^m\|A_iZ-Z^TA_i\|_F^2
\ge \sum_{i=1}^m |x^*(A_iZ-Z^TA_i)x|^2\\
&=\sum_{i=1}^m |\lambda-\bar{\lambda}| \ |x^*A_i x|^2
=|\lambda-\bar{\lambda}|  \sum_{i=1}^m  |x^*A_i x|^2.
\end{align*}
The conclusion follows.
\qquad\end{proof}

Theorem~\ref{rp} tells that when $\delta$ is small, 
 the imaginary part of any eigenvalue of $Z\in\nad$ will be small.
Consequently, we may cluster the eigenvalues by their real parts only.

\subsection{\textsc{geanogjbd} with a greedy strategy}\label{subsec:alg1}
Using similar arguments as Theorem~\ref{thm:max},
we may claim that for almost all $Z\in\nad$,
$\zeta_{n}^{\opt}(Z)$ is maximized.
So we may determine a $Z\in\nad$ and an eigenvalue partition of it for the {\sc geanogjbd} problem in a greedy way:
\begin{enumerate}
\item Arbitrarily determine a $Z$ from $\nad$, 
say a random linear combination of an orthonormal basis
\footnote{
For any two square matrices  $A$, $B\in\mathbb{R}^{n\times n}$, their inner product is defined as $(A,B)=\tr(A^T B)$.
} of $\nad$;
\item Compute the eigenvalue of $Z$, then determine the eigenvalue partition 
by detecting the gap between the real parts of the eigenvalues.
\end{enumerate}
With this greedy strategy, we propose the greedy algorithm for the {\sc geanogjbd} problem.

\begin{algorithm}
  \caption{ \textsc{geanogjbd-greedy} 
    \label{alg:greedy}}
  \begin{algorithmic}[1]
   \Require{The matrix set $\mathcal{A}=\{A_i\}_{i=1}^m$, a parameter $\gamma$ used to control the approximate null space $\nad$.}
   \Ensure{A partition $\tau_n$, a nonsingular matrix $W$ and $f(\tau_n,W)$. }
  \Statex
\State Compute the SVD of $\mathcal{L}$ as in \eqref{svdl};
\State 
Set $\delta=\gamma\sigma_{n^2-1}$, $\ell=\argmax_j \sigma_{n^2-j+1}<\gamma\sigma_{n^2-1}$;
 For $j=1,\dots, \ell$, let $Z_j=\reshape(v_{n^2-j+1},n,n)$;
\State Set $Z=\sum_{j=1}^{\ell}\alpha_j Z_j$, where $\alpha_j$'s are random real numbers;
\State
Compute the real Schur decomposition of $Z=QTQ^T$,
where the real parts of the eigenvalues $\lambda_1,\dots,\lambda_n$ of $T$ are sorted in an ascending order;
\State
Find indices $i_1,\dots, i_{t-1}$ such that $\Re(\lambda_{i_j+1})-\Re(\lambda_{i_j})\ge \mu (\Re(\lambda_{n})-\Re(\lambda_{1}))$ for all possible $j$;
\State  Set $\tau_n=(n_1,\dots,n_t)$,
where for $j=1,\dots,t$, $n_j=i_j-i_{j-1}$ ($i_0=0,i_{t}=n$);
\State 
Compute the eigenvalue decomposition of $T$ corresponding with $\tau_n$, $T=W\diag(T_1,\dots,T_t)W^{-1}$;
\State For $j=1,\dots,t$, compute the `economic' QR factorization of $W(:,i_{j-1}+1:i_{j})$, i.e.,
$W(:,i_{j-1}+1:i_j)=U_jR_j$,
where $U_j\in\mathbb{R}^{n\times n_j}$, $R_j\in\mathbb{R}^{n_j\times n_j}$.
\State
Compute $W=Q[U_1\,\dots\, U_t]$ and  $f(\tau_n,W)=\sum_{i=1}^m\|\OffBdiag_{\tau_n}(W^T A_i W)\|_F^2$.
  \end{algorithmic}
\end{algorithm}

Several remark follows.
\begin{remark}
\begin{enumerate}
\item On input, the parameter is set as $\gamma=1.2$ in our numerical tests.
\item Step 1, the overall computational cost of \textsc{geanogjbd-greedy} (and also \textsc{geanogjbd-consv}) is dominated by the computation of the SVD of $\mathcal{L}$,
which requires $O(n^6)$ flops.
So when $n$ is large, the algorithm can be slow.
In order to improve the efficiency of the algorithm, it is worth exploring the structure of $\mathcal{L}$ to design efficient methods to compute its SVD.
More work are needed here.

\item Step 2, $\{Z_1,\dots,Z_{\ell}\}$ forms an orthonormal basis of $\nad$.
\item Step 4, in order to make the real parts of the eigenvalues in an ascending order,
a reorder of the eigenvalues is required, which can be done by Algorithm~7.6.1 in \cite{van1996matrix}. 
\item Step 5, the parameter $\mu$ is used to detect the gap of the eigenvalues. 
If the difference between the real parts of two  consecutive eigenvalues $\Re(\lambda_{i+1})-\Re(\lambda_{i})$ is smaller than $\mu (\Re(\lambda_{n})-\Re(\lambda_{1}))$,
we take them as in the same cluster, in two different clusters otherwise.
In our numerical tests, we set $\mu=\frac{1}{8(n-1)}$. 
\item Step 7, the eigenvalue decomposition can be computed via Algorithm~7.6.3 in \cite{van1996matrix}.
\item By computing $W(:,i_{j-1}+1:i_j)=U_jR_j$ in Step 7, $W$ in Step 8 satisfies \eqref{wtau}.
\item Strictly speaking, the solution $(\tau_n,W)$ returned by the algorithm may  not satisfy
$f(\tau_n,W) \le \epsilon^2$.
But according to Theorem~\ref{thm:anogjbd}, $f(\tau_n,W)$ will not be large if $\delta$ is small.
As $\delta$ is set as $\gamma \sigma_{n^2-1}$, it will be reasonably small since $\sigma_{n^2-1}$ is 
the smallest singular vector which corresponds with a nontrivial solution to the {\sc geanogjbd} problem.

\end{enumerate}
\end{remark}

\subsection{\textsc{geanogjbd} with a conservative strategy}\label{subsec:alg2}

Numerically, when $Z$ is arbitrarily chosen from $Z\in\nad$,
it may be  unstable to compute the eigenvalue partition of $Z$ 
corresponding with $\zeta_n^{\opt}(Z)$.
In order to deal with such instability, we prefer to find the partition in a ``conservative'' way:
In the first step, find a ``good'' $Z$ in $\nad$ in the sense that
the real parts of eigenvalues of $Z$ can be split into two clusters and the gap between these two clusters are relatively large.
Compute an eigenvalue decomposition of $Z$ with two diagonal blocks,
each block corresponds with a cluster of eigenvalues.
Approximately block diagonalizing $A_i$ by a congruence transformation
(the transformation matrix is nothing but the similarity transformation matrix in the eigenvalue decomposition of $Z$),
then the original {\sc geanogjbd} problem can be decoupled 
into two separate  {\sc geanogjbd} problems with smaller sizes.
For each smaller problem, we can perform the above procedure recursively.

From the above discussion, we can see that the key step is to find a ``good''  $Z$.
Next we show that there is a lower bound for the maximum gap between the real parts of the eigenvalues.
\begin{theorem}\label{thm3}
Let $\theta_1\le \dots\le \theta_n$ be the real parts of the eigenvalues of $Z\in\nad$.
If $\tr(Z)=0$ and $\tr(Z^2)=\eta\ge 0$,
then
\[
g:=\max_{1\le j\le n-1}|\theta_{j+1}-\theta_j|\ge \sqrt{\frac{8\eta}{(n-1)n^2}}.
\]
\end{theorem}

\begin{proof}
Let $d_0=0$, $d_{j}=\theta_{j+1}-\theta_j$ for $j=1,\dots,n-1$.
Using $\tr(Z)=\sum_{i=1}^n\theta_j=0$, we have
\[
\theta_1=-\frac{\sum_{i=1}^{n-1} (n-i)d_i}{n}.
\]
Let $d=[d_1 \ \dots \ d_{n-1}]^T$, define $f(d):=\sum_{i=0}^{n-1}(\theta_1+d_1+\dots+d_i)^2$.
By calculations, we have
\begin{align*}
\frac{\partial f}{\partial d_j}&=2\sum_{i=j}^{n-1}(\theta_1+d_1+\dots+d_i)-
2\sum_{i=0}^{n-1}(\theta_1+d_1+\dots+d_i)\frac{n-j}{n},\\
\frac{\partial^2 f}{\partial d_j\partial d_k}&=
2\sum_{i=\max\{j,k\}}^{n-1}1-2\sum_{i=k}^{n-1}\frac{n-j}{n}-
2\sum_{i=j}^{n-1}\frac{n-k}{n}+
2\sum_{i=0}^{n-1}\frac{n-k}{n}\frac{n-j}{n}\\
&=2(\min\{j,k\}-\frac{jk}{n})\ge 0.
\end{align*}
One one hand, let $f_{jk}=\min\{j,k\}-\frac{jk}{n}$, $F=[f_{jk}]$, then
for any $1\le j\le n-1$, we have
\begin{align*}
\sum_{k=1}^{n-1}f_{jk}=\sum_{k=1}^{n-1}\min\{j,k\}-\frac{jk}{n}=\frac{j(n-j)}{2}\le \frac{n^2}{8},
\end{align*}
and hence $\|F\|_{\infty}\le \frac{n^2}{8}$.
On other hand, notice that
\begin{align*}
\eta=\tr(Z^2)\le \sum_{i=1}^n\theta_i^2=f(d)=d^T F d\le \|F\|_2 d^T d\le (n-1)\|F\|_2 g^2.
\end{align*}
Consequently,
\begin{align*}
g\ge \sqrt{\frac{\eta}{{(n-1)\|F\|_2}}}\ge \sqrt{\frac{\eta}{{(n-1)\|F\|_{\infty}}}}\ge \sqrt{\frac{8\eta}{(n-1)n^2}}.
\end{align*}
This completes the proof.
\qquad\end{proof}

From the above theorem, we know that $g$ will not be small if $\tr(Z^2)$ is not small and $\tr(Z)=0$.
The next theorem shows how to determine a $Z$ such that  $\tr(Z^2)$ is maximized.
\begin{theorem}\label{thm4}
Let $I_n, Z_1,\dots, Z_{\ell}$ be an orthonormal basis of $\nad$.
 Then for any
$Z=\sum_{j=1}^{\ell}\alpha_j Z_j$, it holds that $\tr(Z)=0$, and
$\tr(Z^2)$ is maximized if $\alpha=[\alpha_1\,\dots\,\alpha_{\ell}]^T$ is the eigenvector of $H$ corresponding with its largest eigenvalue,
where $H=[h_{jk}]$ with $h_{jk}=\tr (Z_j Z_k)$.
\end{theorem}

\begin{proof}
First, $\tr(Z)=0$ follows from the fact that $\tr(Z_j)=(I_n,Z_j)=0$ for $j=1,\dots,\ell$.
Second, simple calculation gives
\begin{align*}
\tr(Z^2)=\sum_{j, k}\alpha_j\alpha_k h_{jk}
=\alpha^T H\alpha.
\end{align*}
Noticing that $H$ is symmetric, the conclusion follows.
\end{proof}

Summarizing the above discussions on the ``conservative'' way gives rise to 
the conservative algorithm for the {\sc geanogjbd} problem.
First, we illustrate how to perform one step of the algorithm by function {\sc geanogjbd1step} in algorithm~\ref{alg:conservative1}.

\begin{algorithm}
  \caption{One step of the \textsc{geanogjbd-consv} 
    \label{alg:conservative1}}
  \begin{algorithmic}[1]
   \Function{$[\tau_n,W,f]=$geanogjbd1step}{$\mathcal{A}$,\;$\gamma$}
  \Statex
\State Compute $Z_1,\dots,Z_{\ell}$ such that $Z_1,\dots,Z_{\ell}$ together with $I_n$ form an orthonormal basis of $\nad$;
\State Compute the matrix $H=[h_{jk}]$ with $h_{jk}=\tr (Z_j Z_k)$;
\State Compute $\alpha=[\alpha_1\,\dots\, \alpha_{\ell}]^T$ with $\|\alpha\|_2=1$, where $\alpha$ is the eigenvector of $H$ corresponding with its largest eigenvalue;
\State Set $Z=\sum_{j=1}^{\ell}\alpha_j Z_j$;
\State
Compute the real Schur decomposition of $Z=QTQ^T$,
where the real parts of the eigenvalues $\lambda_1,\dots,\lambda_n$ of $T$ are sorted in an ascending order;
\State
Find the index $n_1= \argmax_{1\le i\le n-1}(\Re(\lambda_{i+1})-\Re(\lambda_{i}))$, set $n_2=n-n_1$,
$\tau_n=(n_1,n_2)$;
\State
Compute the eigenvalue decomposition of $T$ corresponding with $\tau_n$, $T=W\diag(T_1,T_2)W^{-1}$;

\State Compute the `economic' QR factorizations of $W(:,1:n_1)$ and $W(:,n_1+1,n_1+n_2)$, i.e., $W(:,1:n_1)=U_1R_1$, $W(:,n_1+1:n_1+n_2)=U_2R_2$,
where for $j=1,2$ $U_j\in\mathbb{R}^{n\times n_j}$, $R_j\in\mathbb{R}^{n_j\times n_j}$;

\State
Compute $W=Q[U_1\, U_2]$ and $f=\sum_{i=1}^m\|\OffBdiag_{\tau_n}(W^T A_i W)\|_F^2$.
\EndFunction
  \end{algorithmic}
\end{algorithm}

Several remarks follow in order.
\begin{remark}
\begin{enumerate}
\item Line 1, the input are
 the matrix set $\mathcal{A}=\{A_i\}_{i=1}^m$ and a parameter $\gamma$,
which is used to control the
approximate null space $\nad$;
the output are a partition $\tau_n=(n_1,n_2)$, a matrix $W\in\mathbb{R}^{n\times n}$ and $f=f(\tau_n,W)$.
\item Line 2, an orthonormal basis of $\nad$ can be obtained in the same way as in Steps 1 and 2 of algorithm~\ref{alg:greedy}. Then $Z_1,\dots,Z_{\ell}$ can be obtained via modified Gram-Schmidt process. 
\item Line 3 to 5, determine a $Z\in\nad$ with $\tr(Z^2)$ maximized (Theorem~\ref{thm4}). 
\end{enumerate}
\end{remark}

\bigskip

Next, we are ready to present  the \textsc{geanogjbd-consv} algorithm in algorithm~\ref{alg:conservative}.

\begin{algorithm}
  \caption{\textsc{geanogjbd-consv}
    \label{alg:conservative}}
  \begin{algorithmic}[1]
   \Require{The matrix set $\mathcal{A}=\{A_i\}_{i=1}^m$, the tolerance $\epsilon$ and a parameter $\gamma$ used to control the approximate null space $\nad$. }
   \Ensure{A solution $(\tau_n,W)$ to the {\sc geanogjbd} problem and $f=f(\tau_n,W)$. }
  \Statex

\State Set $\tau_n=(n)$, $W=I_n$, $f=0$;
\State Call $[\phi,W_d,v_f]=\text{\sc geanogjbd1step}(\{A_i\}_{i=1}^m, \; \gamma)$;
\State Set $\ell=1$, $p=1$, $q=n$, $\tau=\phi(1)$, $\hat{\tau}_n=\phi$, $\widehat{W}=W_d$, $\hat{f}=v_f$;
\While{$\hat{f}\le \epsilon^2$}
\State $\tau_n=\hat{\tau}_n, W=\widehat{W}, f=\hat{f}$;
\State $W_1=W(:, p:p+\tau(\ell)-1)$, $W_2=W(:, p+\tau(\ell):q)$;
\State For $i=1,\dots,m$, compute $A_{i1}=W_1^TA_iW_1$, $A_{i2}=W_2^TA_iW_2$;
\State For $j=1,2$, call $[\phi_j,W_{jj},f_j]=\text{\sc geanogjbd1step}(\{A_{ij}\}_{i=1}^m,\; \gamma)$;
\State Set $W_d(p:q,p:q)=\diag(W_{11}, W_{22})$, $\tau(\ell)=(\phi_1(1),\phi_2(1))$, $v_f(\ell)=[f_1 \; f_2]$;
\State $\ell=\argmin_j v_f(j)$;
\State $p=\sum_{j=1}^{\ell-1}\tau_n(j)+1$, $q=p+\tau_n(\ell)-1$;
\State $\hat{\tau}_n(\ell)=(\tau(\ell), \tau_n(\ell)-\tau(\ell))$;
\State $\widehat{W}(:, p:q)=W(:, p:q)W_d(p:q, p:q)$;
\State Compute $\hat{f}=\sum_{i=1}^m\|\OffBdiag_{\hat{\tau}_n}(\widehat{W}^T A_i \widehat{W})\|_F^2$.
\EndWhile

  \end{algorithmic}
\end{algorithm}

Several remarks follow in order.
\begin{remark}
\begin{enumerate}
\item Let $(\tau_n, W, f)$ be the current guess of the solution,
where $\tau_n=(n_1,\dots,n_s)$, $W=[W_1\dots W_s]$ with
$W_j\in\mathbb{R}^{n\times n_j}$ for $j=1,\dots, s$, and $f=f(\tau_n,W)$.
For each $1\le j\le s$, the output of the function {\sc geanogjbd1step} with input $\{W_j^TA_iW_j\}_{i=1}^m$ is $(\phi_j,W_{jj},f_j)$.
Then $\tau$ stores the first elements of $\phi_j$'s in a vector,
$W_d$ stores $W_{jj}$'s in a block diagonal matrix,
$v_f$ stores $f_j$'s in a vector.
To be specific, $\tau=(\phi_1(1),\dots,\phi_s(1))$,
$W_d=\diag(W_{11},\dots,W_{ss})$,
$v_f=[f_1\,\dots\,f_s]$.
The integer $\ell$ stores the index $\argmin_jv_f(j)$,
integers $p$ and $q$ stores the first and last row indices of the $\ell$th diagonal block of $W_d$, respectively.
The triple $(\hat{\tau}_n, \widehat{W}, \hat{f})$ stores the next guess of the solution.

\item Notice that the solution $(\tau_n,W)$ returned by {\sc geanogjbd-consv} will satisfy
$f(\tau_n,W)\le \epsilon^2$, unlike {\sc geanogjbd-greedy}.
\end{enumerate}
\end{remark}

\section{Numerical Experiments}\label{sec:numer}
Now we present several numerical examples to illustrate the performance of our methods. 
All the numerical examples were carried out using MATLAB R2014b, 
with machine $\epsilon=2.2\times 10^{-16}$.
We compare the performance of our algorithms with four other algorithms for the {\sc geanojbd} problem, namely,  JBD-OG, JBD-ORG \cite{ghennioui2010gradient}, 
JBD-LM \cite{cherrak2013non} and JBD-NCG \cite{nion2011tensor}.
For the JBD-OG method and  the JBD-ORG method
the stopping criteria are 
$\|W_{k+1}-W_k\|_F<10^{-12}$, or $\dfrac{\phi_k - \phi_{k+1}}{\phi_k}<10^{-8}$ for successive 5 steps, 
or the maximum number of iterations,  which is set as 2000, exceeded.
Here $W_k$, $\phi_k$ are the $W$ matrix and the value of the cost function in $k$th step, respectively. 
For the JBD-LM method,
the stopping criteria are the same as that of the JBD-OG/ORG method,
except the maximum number of iterations is set as 200. 
And in order to avoid degenerate solutions, we use \eqref{wtau} to normalize $W_k$ in each step.
For the JBD-NCG method,
the stopping criteria are $\phi_k-\phi_{k+1}<10^{-8}$ 
or $\frac{\phi_k - \phi_{k+1}}{\phi_k}<10^{-8}$, 
or the maximum number of iterations, which is set as 2000, exceeded.
In all four iterative algorithms above, 20 initial values (19 random initial values and an EVD-based initial value \cite{nion2011tensor}) are used to iterate 20 steps first, 
and then the iteration which produces the smallest value of the cost function proceeds until one of the stopping criteria is satisfied.

\subsection{Random data}\label{subsec:random data}
Let $\tau_n=(n_1, \dots, n_t)\in\tn$, we use the model in \cite{ghennioui2010gradient} to generate the matrix set $\{A_i\}_{i=1}^m$:
\begin{align}\label{modelVDV}
A_i=V^T D_i V, \quad i=1,\dots,m,
\end{align}
where $V$, $D_i$'s are, respectively, the mixing matrix and the approximate $\tau_n$-block diagonal matrices.  The elements in $V$ and $\Bdiag_{\tau_n}(D_i)$ are all real numbers drawn from a standard normal distribution, while the elements in $\OffBdiag_{\tau_n}(D_i)$ are all real numbers drawn from a normal distribution with mean zero and variance $\sigma^2$. The signal-to-noise ratio is defined as $\SNR=10\log_{10}(1/\sigma^2)$. 
In the {\sc geanogjbd-consv} algorithm, the parameter $\epsilon$ is set as 
$\epsilon=3n^2 10^{-\SNR/20}$.

For model \eqref{modelVDV}, we define the same performance index defined in \cite{cai2015matrix} to measure the quality of the computed solution $W$:
\begin{align}\label{pi}
\PI(V^{-1}, W)=\min_\pi\max_{1\leq i\leq t}\subspace(V_i, W_{\pi(i)}),
\end{align}
where $V^{-1}=[V_1 \, \dots \, V_t], \; W=[W_1 \, \dots \, W_t], \; V_i, \; W_{\pi(i)}\in\mathbb{R}^{n\times n_i}, \; i=1, \dots, t$, the vector $\pi=(\pi(1), \dots, \pi(t))$ is a permutation of $\{1, \dots, t\}$ satisfying $(n_{\pi(1)}, \dots, n_{\pi(t)})=\tau_n$, and the expression $\subspace(E, F)$ denotes the angle between two subspaces specified by the column vectors of $E$ and $F$, which can be computed by the MATLAB function "subspace". The smaller the performance index is, the better $W$ is.
In order to make a fair comparison of the algorithms,
the matrices $W$ returned by different algorithms are normalized to satisfy \eqref{wtau}.

Let $(\hat{\tau}_n, \widehat{W})$ be a solution returned by one of our algorithms
with $\hat{\tau}_n=(\hat{n}_1,\dots,\hat{n}_{\hat{t}})$.
The partition $\hat{\tau}_n$ returned by our algorithms may be not equivalent to $\tau_n$ in \eqref{modelVDV}, especially when the SNR is small.
But we still say that $\hat{\tau}_n$ is correct  if $\card(\hat{\tau}_n)\ge \card(\tau_n)$ and there exists
a $(0, 1)$ matrix $N$ such that $\tau_n=\hat{\tau}_n N$.
Notice that such $(0,1)$ matrix $N$ is not unique.
Once $N$ is fixed, the columns of $\widehat{W}$ need to be reordered accordingly.
Denote the resulting matrix as $\widetilde{W}_{N}$,  the performance index  $\PI(V^{-1},\widetilde{W}_N)$ can be defined as in \eqref{pi}.
Then the performance index of $(\hat{\tau}_n, \widehat{W})$ with $\hat{\tau}_n$ being correct, can be defined as $\min_{N}\PI(V^{-1}, \widetilde{W}_N)$.
For example,   let $\tau_n=(1,2,3)$, $\hat{\tau}_n=(2,4)$ or $\hat{\tau}_n=(1,1,2,2)$.
In the former  case, $\hat{\tau}_n$ is not correct, we say that our algorithm fails.
In the latter case, $\hat{\tau}_n$ is correct.
The $(0,1)$ matrix can be one of the following matrices:
\begin{align*}
N_1=\bsmat
1 & 0 & 0 \\
0 & 0 & 1\\
0 & 1 & 0\\
0 & 0 & 1
\esmat,
\quad
N_2=\bsmat
0 & 0 & 1 \\
1 & 0 & 0\\
0 & 1 & 0\\
0 & 0 & 1
\esmat,
\quad
N_3=\bsmat
1 & 0 & 0 \\
0 & 0 & 1\\
0 & 0 & 1\\
0 & 1 & 0
\esmat,
\quad
N_4=\bsmat
0 & 0 & 1 \\
1 & 0 & 0\\
0 & 0 & 1\\
0 & 1 & 0
\esmat,
\end{align*}
and the corresponding $\widetilde{W}_N$ can be given by
\begin{align*}
\widetilde{W}_{N_1}=\widetilde{W}( :, [1\, 3 \, 4 \, 2 \, 5 \, 6]),\quad
\widetilde{W}_{N_2}=\widetilde{W}( :, [2\, 3 \, 4 \, 1 \, 5 \, 6]),\\
\widetilde{W}_{N_3}=\widetilde{W}( :, [1\, 5 \, 6 \, 2 \, 3 \, 4]),\quad
\widetilde{W}_{N_4}=\widetilde{W}( :, [2\, 5 \, 6 \, 1 \, 3 \, 4]).
\end{align*}
Then the performance index for $(\hat{\tau}_n,\widehat{W})$
is $\min_{i=1,2,3,4}\PI(V^{-1},\widetilde{W}_{N_i})$.

We generate the matrix sets by model \eqref{modelVDV} with the following parameters:

\textbf{Case 1.} Let $n=9, \tau_n=(3,3,3), m=20$.

\textbf{Case 2.} Let $n=10, \tau_n=(1,2,3,4), m=20$.

For different SNRs, we compare our algorithms with the above four algorithms 
in terms of performance index.
And for each SNR, we perform 50 independent trials. 
The box plot (generated by MATLAB function ``boxplot'') of the results are displayed in Figure~\ref{fig:case1} and \ref{fig:case2}.

\begin{figure}[!hbt]
\centering
\includegraphics[width=\textwidth]{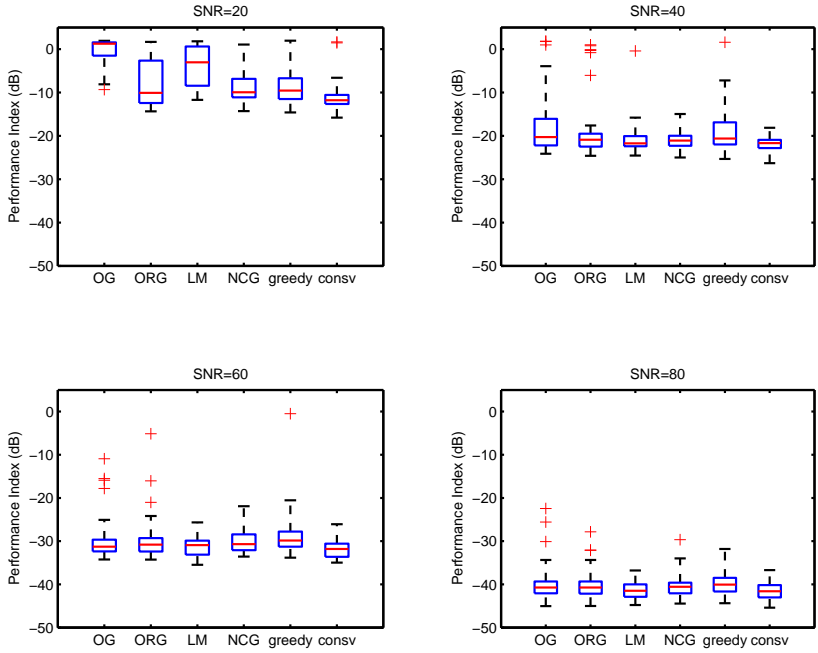}
\caption{Performance of different algorithms: Case 1}
\label{fig:case1}
\end{figure}

\begin{figure}[!hbt]
\centering
\includegraphics[width=\textwidth]{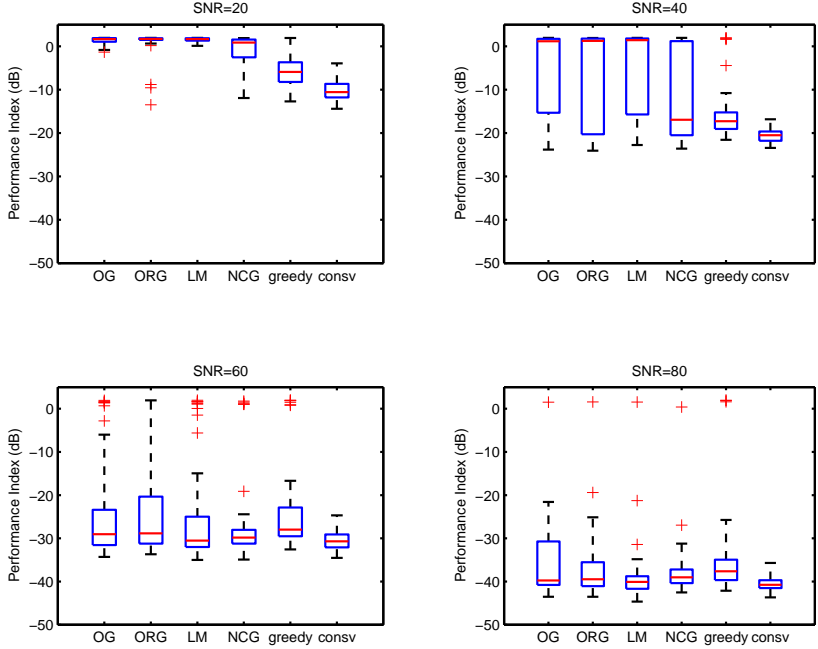}
\caption{Performance of different algorithms: Case 2}
\label{fig:case2}
\end{figure}

In case 1, the sizes of diagonal blocks are the same.
We can see from Figure~\ref{fig:case1} that when SNR equals to 40, 60 or 80,
the performance indices produced by all six algorithms are almost the same on average, but the {\sc geanogjbd-consv} method is more robust;
when SNR equals to 20, the performance indices produced by the JBD-ORG method, the JBD-NCG method and our two methods are almost the same, and smaller than 
those of the other two methods.
In case 2, the sizes of diagonal blocks are different.
We can see from Figure~\ref{fig:case2} that when SNR equals to 60 or 80,
the performance indices produced by all six algorithms are almost the same on average, and the {\sc geanogjbd-consv} method is obviously more robust than the other five methods;
when SNR equals to 20 or 40, the {\sc geanogjbd-consv} method produces the smallest performance index.

\subsection{Separation of convolutive mixtures of source}
We consider example 5.2 in \cite{cai2015matrix},
all settings are kept the same except that 
\begin{enumerate}
\item
The source signals are mixed according to the transfer function matrix given by:
\begin{align*}
H[z]=& \bsmat
0.59 & -0.67 & 0.89 \\
0.42 & 0.42 & -0.92 \\
0.41 & -0.91 & -0.34 \\
0.14 & -0.44 & 0.86 \\
0.04 & -0.73 & -0.95 \\
-0.86 & -0.04 & 0.41
\esmat+
\bsmat
-0.89 & 0.80 & -0.78 \\
0.65 & -0.25 & 0.58 \\
0.99 & 0.19 & 0.48 \\
-0.75 & -0.85 & 0.99 \\
-0.70 & -0.76 & 0.81 \\
0.32 & -0.23 & -0.19
\esmat z^{-1}\\
& +\bsmat
0.40 & -0.12 & -0.51 \\
-0.55 & 0.71 & 0.88 \\
-0.28 & -0.71 & -0.41 \\
0.64 & 0.93 & 0.63 \\
0.43 & -0.70 & -0.98 \\
0.72 & 0.53 & -0.03
\esmat z^{-2}+
\bsmat
-0.07 & 0.69 & -0.05 \\
0.45 & 0.56 & -0.94 \\
0.24 & -0.29 & 0.42 \\
-0.02 & -0.21 & -0.89 \\
-0.03 & 0.63 & -0.21 \\
0.06 & -0.73 & -0.55
\esmat z^{-3},
\end{align*}
where $H[z]$ stands for the $z$ transform of $H(t)$;

\item 
The {\sc geanogjbd} problem for the correlation matrix set $\{R_Y(t,\tau)\}_{\tau=0}^2$
is considered rather than the {\sc heanogjbd} problem for $\{\frac{R_Y(t,\tau)+R_Y(t,\tau)^T}{2}\}_{\tau=0}^2$;
\item
After the source signal vectors is recovered by $\hat{S}(t)=B^TX(t)$,
where $B$ is a solution to the {\sc geanogjbd} problem, $X(t)$ is the observed signal vector, a blind SIMO system identification step via the subspace-based technique proposed in \cite{moulines1995subspace} is applied to obtain  the primary sources signals $\hat{s}_1(t), \hat{s}_2(t), \hat{s}_3(t)$;
\end{enumerate}

The maximum correlation between the $i$th source signal $s_i(t)$ and the recovered signals $\hat{s}_1(t), \hat{s}_2(t), \hat{s}_3(t)$, i.e., $\max_{j=1,2,3}$ correlation $(s_i(t),\hat{s}_j(t))$,
is used to  estimate the quality of the $i$th recovered signal.
The larger the correlation is, the better the source signal is recovered.

In Figure~\ref{fig:signal}, we plot the correlations between the source signals and the extracted signals obtained from computed solutions by different algorithms for $s_1(t), s_2(t), s_3(t)$, respectively. 
All displayed results have been averaged over 50 independent trials,
and in the {\sc geanogjdb-consv} method, we set $\epsilon=0.05$ for all SNRs.

\begin{figure}[!hbt]
\centering
\includegraphics[width=0.45\textwidth]{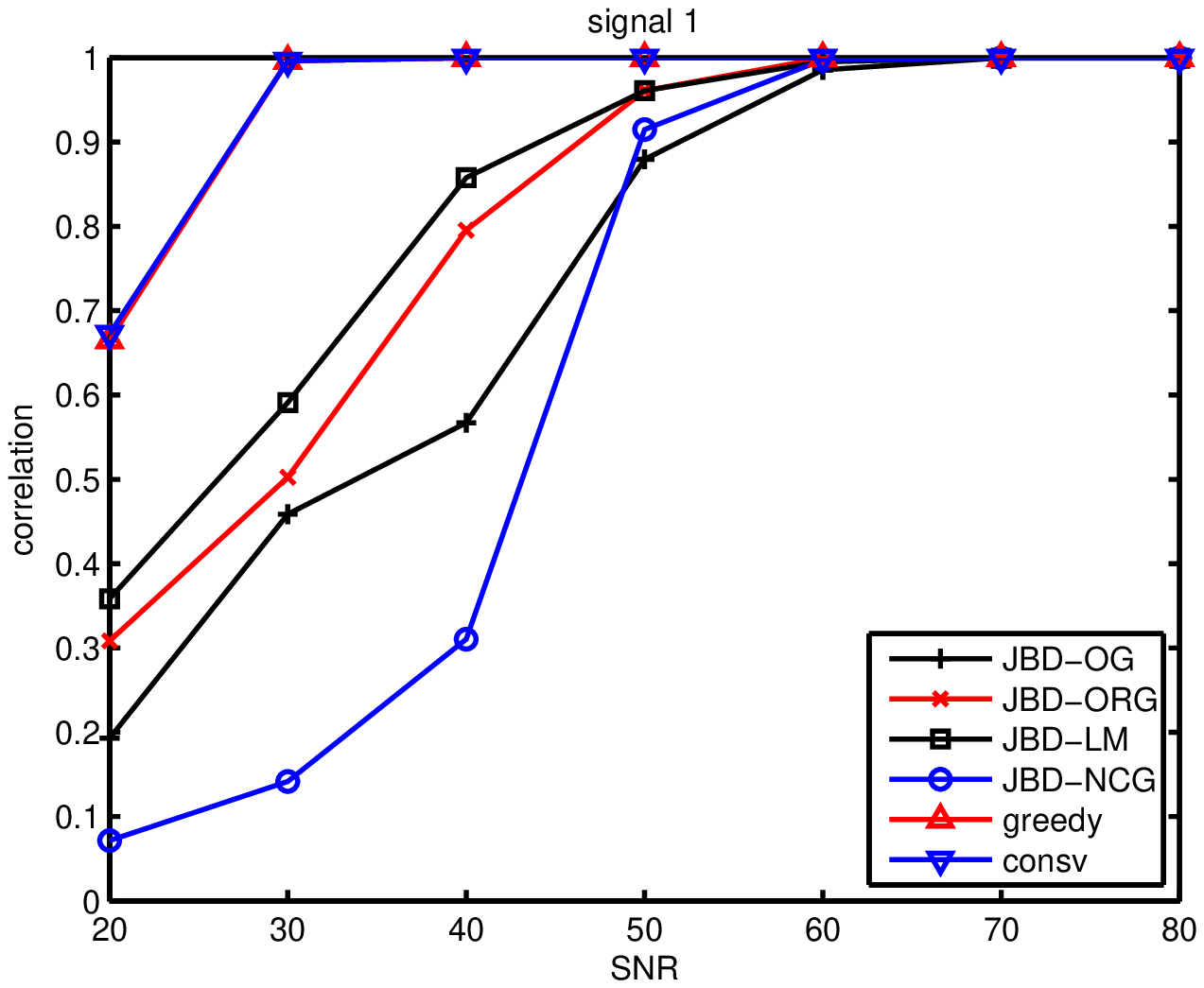}
\includegraphics[width=0.45\textwidth]{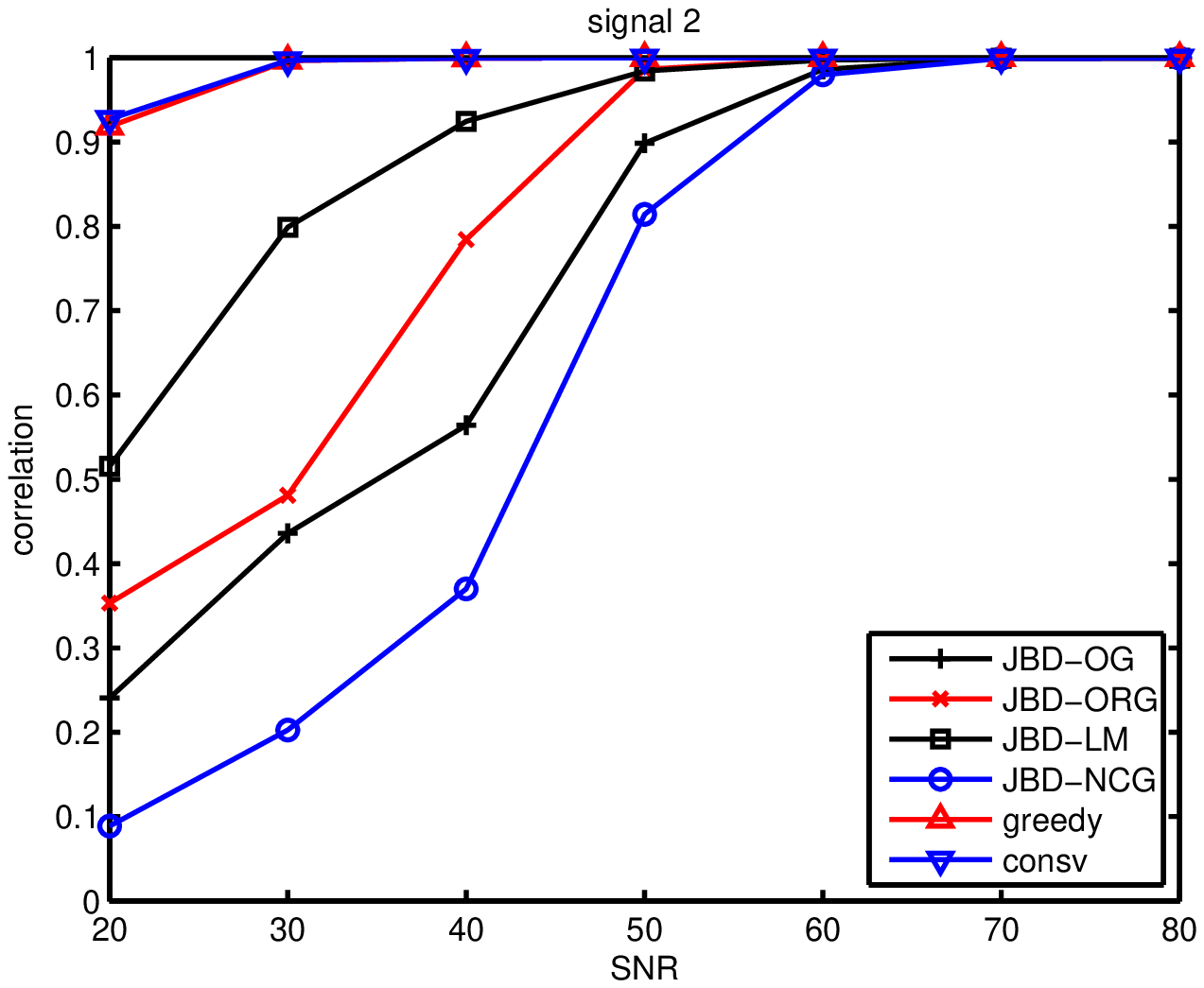}\\
\includegraphics[width=0.45\textwidth]{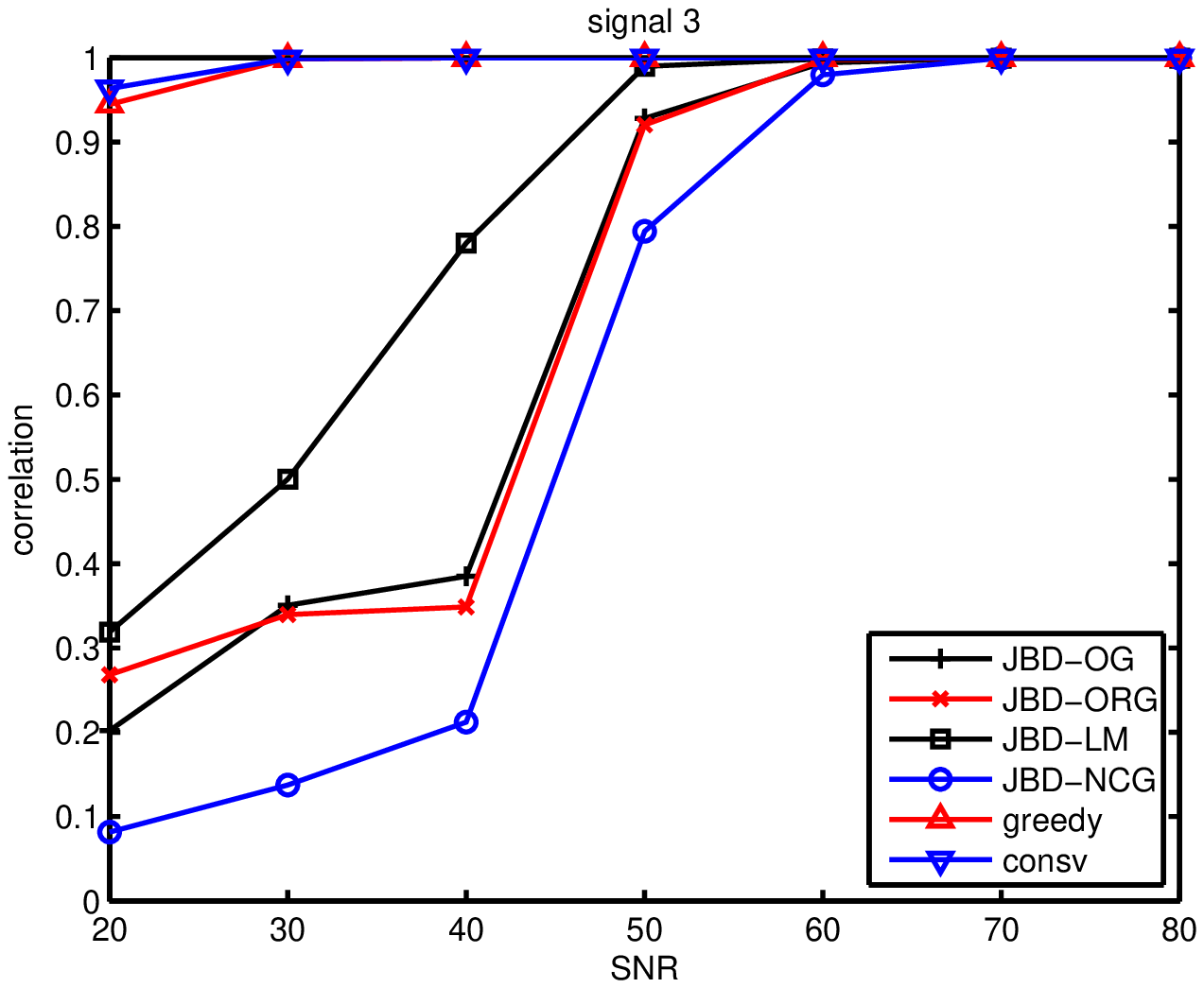}
\caption{ Correlation between recovered signals and source signals}
\label{fig:signal}
\end{figure}

It can be seen from Figure~\ref{fig:signal} that when SNR is more than 50,
the recovered signals obtained from all six algorithms are at the same level of quality, and when SNR is smaller than 40, our algorithms give much better result.

\section{Conclusion}\label{sec:conclusion}
In this paper, we show that the solution to the {\sc geenogjbd} problem or the {\sc geanogjbd} problem can be obtained by finding a proper $Z$ in $\na$ or $\nad$ followed by computing an eigenvalue decomposition of $Z$.
A necessary and sufficient condition for the equivalence of all solutions to the {\sc geenogjbd} problem is established.
Based on the established theory, two algorithms are proposed to solve the {\sc geanogjbd} problem. The first algorithm, which uses a greedy strategy,
is simple and efficient, but may suffer from instability.
The second algorithm, which uses a conservative strategy,
is an iterative method, will terminate within finite steps,
and is much more stable than the first algorithm.
Our limited numerical experiments show that the {\sc geanogjbd-consv} method outperforms the current iterative algorithms based on optimization, especially
when the SNR is small.

It is also worth mentioning here that the necessary and sufficient condition for the equivalence of all solutions to the {\sc geenogjbd} problem can be used to analyze
the sensitivity of the {\sc geenojbd} problem, and we will present the results in a separate paper.
The approach we treat the {\sc geenogjbd}/{\sc geanogjbd} problem
 in this paper can also be used to deal with the {\sc geenugjbd}/{\sc geanugjbd} problem. 

Finally, noticing that, 
compared with numerical methods for the BTD of tensors (see e.g., \cite{de2008decompositions3,tensorlab}), 
the {\sc gjbd} problem present in this paper is limited in several aspects:
the matrices $A_i$'s are square rather than general non-square ones;
the matrices $A_i$'s are factorized via a congruence transformation: 
$A_i=W^{-T}\diag(A_i^{(jj)})W^{-1}$ (see \eqref{eq:nojbd}),
rather than a general factorization: $A_i=U D_i V^T$, 
where $D_i$'s are block diagonal, $U$, $V$ are not necessarily square.
A natural question is that can we adopt a similar algebraic approach in this paper to 
remove these limitations?
To that end, more work are obviously needed, and our initial results seem inspiring.

\noindent{\bf Acknowledgement. } The authors wish to give thanks to 
Dr. O.~Cherrak and Dr. D.~Nion for sharing their Matlab codes of 
the JBD-LM method and the JBD-NCG algorithm, respectively.
The authors also want to give thanks to the anonymous referees 
for their comments and suggestions, which help us improve the paper.

\bibliographystyle{abbrv} 

\end{document}